\documentclass[11pt]{amsart}
\usepackage[english]{babel}
\usepackage[latin1]{inputenc}
\usepackage{tikz}\usetikzlibrary{calc}
\usepackage{tkz-euclide}
\usepackage{amssymb,geometry,color,xcolor,graphicx}
\usepackage{amsmath,amsthm}
\usepackage{booktabs}
\usepackage{caption}
\usepackage{subcaption}
\usepackage{datetime}
\usepackage{dsfont}
\usepackage{amstext}
\usepackage{graphicx}
\usepackage{fancyhdr}
\usepackage{latexsym}
\usepackage{mathrsfs}
\usepackage{cases}
\usepackage{mathtools}
\usepackage{bm}
\usepackage{hyperref}
\hypersetup{
    colorlinks=true, %set true if you want colored links
    linktoc=all,     %set to all if you want both sections and subsections linked
    linkcolor=blue,  %choose some color if you want links to stand out
}
\numberwithin{equation}{section}

\newcommand{\sref}[1]{Section~\ref{#1}}

\newcommand{\udef}{\mathrel{\mathop:}=}

\newcommand{\R}{\mathbb{R}}

\usepackage{dsfont}
\newcommand{\1}{\mathds{1}}
\newcommand{\sgn}{\mathrm{sgn}}
\newcommand{\smat}[1]{ \left[\begin{smallmatrix} #1 \end{smallmatrix}\right]}
\newcommand{\bmat}[1]{ \begin{bmatrix}#1\end{bmatrix}}
\newcommand{\abs}[1]{\lvert#1\rvert}

\newcommand{\conv}{\textrm{conv}\,}

\renewcommand{\Re}{\mathbb{R}}
\newcommand{\gbc}{\varphi}

\newtheorem{thm}{Theorem}[section]

\newtheorem{prop}[thm]{Proposition}
\newtheorem{lem}[thm]{Lemma}

\newtheorem{rem}[thm]{Remark}
\newtheorem{exm}[thm]{Example}

\numberwithin{equation}{section}

\makeatletter
\renewcommand{\@biblabel}[1]{#1\hfill \hspace{-0.2cm}}
\makeatother

\graphicspath{ {./figs/}
}

\begin{document}

\title{Nonnegative moment coordinates on finite element geometries}
\author[Dieci]{Luca Dieci}
\address{School of Mathematics, Georgia Institute of Technology,
Atlanta, GA 30332 U.S.A.}
\email{dieci@math.gatech.edu}
\author[Difonzo]{Fabio V. Difonzo}
\address{Istituto per le Applicazioni del Calcolo \textquotedblleft Mauro Picone\textquotedblright, Consiglio Nazionale delle Ricerche, Via G. Amendola 122/I, 70126 Bari, ITALY}
\email{fabiovito.difonzo@cnr.it}
\author[Sukumar]{N. Sukumar}
\address{Department of Civil and Environmental Engineering, One Shields Avenue, University of California, Davis, CA 95616, USA}
\email{nsukumar@ucdavis.edu}



\null\hfill Version of \today $, \,\,\,$ \xxivtime

\begin{abstract}
In this paper, we introduce new generalized barycentric coordinates
(coined as {\em moment coordinates}) on convex and nonconvex quadrilaterals and convex hexahedra with planar faces.
This work draws on recent
advances in constructing interpolants to
describe the motion of the Filippov sliding vector field in nonsmooth
dynamical systems, in which nonnegative
solutions of signed matrices based on (partial) distances are studied.
For a finite
element with $n$ vertices (nodes) in $\Re^2$, the
constant and linear reproducing conditions are supplemented
with additional linear moment equations to set up a linear system of equations
of full rank $n$, whose solution results in the nonnegative shape functions.
On a simple (convex or nonconvex)
quadrilateral, moment coordinates using signed distances
are identical to mean value coordinates. For signed weights that are based on
the product of distances to edges that are incident to a vertex and their edge lengths, we recover Wachspress coordinates on a convex quadrilateral.
Moment coordinates
are also constructed on a convex hexahedra
with
planar faces.  We present proofs in support of the construction and plots of the shape functions that
affirm its properties.
\end{abstract}


\maketitle

\pagestyle{myheadings}
\thispagestyle{plain}
\markboth{DIECI, DIFONZO AND SUKUMAR}{NONNEGATIVE MOMENT COORDINATES ON FINITE ELEMENT GEOMETRIES}

\section{Introduction}\label{sec:intro}
Generalized barycentric coordinates (GBCs)~\cite{Floater:2015:GBC,Anisimov:2017:BCA} are used to
interpolate data that are prescribed at the vertices of a polytope. These coordinates (shape functions) extend the well-known barycentric coordinates on simplices to
arbitrary polytopes in $\Re^d$ with more than $d+1$ vertices. GBCs have found applications in computer graphics, computational mechanics, and the applied
% sciences~\cite{Hormann:2017:GBC,ArbogastEtAl:2022:NumMath}.
sciences~\cite{Hormann:2017:GBC}.
Let $P \subset \Re^d$ be a
polytope (convex or nonconvex) with $n$ vertices $\{ v_i \}_{i=1}^n $.
A point in $P$ is denoted by $p \in \Re^d$. In its most general definition,
generalized barycentric coordinates, $\gbc(p) := \{ \gbc_i \}_{i=1}^n$, must satisfy the following
relations:
\begin{equation}\label{eq:gbc}
\sum_{i=1}^n \gbc_i(p) = 1, \quad
\sum_{i=1}^n \gbc_i(p) v_i = p, \quad \gbc_i(v_j) = \delta_{ij}, \quad \gbc_i \ge 0,
\end{equation}
where $\delta_{ij}$ is the Kronecker-delta and~\eqref{eq:gbc} ensures that
the generalized barycentric interpolant has linear precision. In addition,
it is also desirable that $\gbc$ are smooth in the interior of $P$.

Wachspress~\cite{Wachspress:2016:RBG} was the first to introduce nonnegative rational basis that were valid for convex polygons;
Warren~\cite{Warren:1996:BCC} extended it for convex polyhedra.
Floater~\cite{Floater:2003:MVC} proposed the mean value coordinates, which have a simple expression that is amenable to numerical computations, and is the most popular and widely used set of GBCs. These coordinates are valid for simple (convex and nonconvex) and nested polygons, but the coordinates can become negative on nonconvex polygons~\cite{Hormann:2006:MVC}.  For applications such as shape deformation in computer graphics,
data approximation over meshes that preserve the maximum principle and mitigate the Runge phenomenon, and to ensure positivity
of the mass matrix entries in an explicit Galerkin method for
elastodynamic simulations, it is beneficial that the nonnegativity condition on $\gbc$ is met; however,
over the years, many coordinates have been proposed that can become negative in $P$~\cite{Anisimov:2017:BCA}.
On convex polytopes, coordinates that are nonnegative
are endowed with the facet-reducing property on its
faces~\cite{Arroyo:2006:LME}.
The only coordinates that are nonnegative and can
be $C^1$ smooth in $P$ are harmonic coordinates~\cite{Joshi:2007:HCF},
maximum-entropy coordinates~\cite{Sukumar:2007:OAC,Hormann:2008:MEC}, local barycentric coordinates~\cite{Zhang:2014:LBC,Tao:2019:FNS},
blended barycentric coordinates~\cite{Anisimov:2017:BBC} and iterative coordinates~\cite{Deng:2020:IC}. Harmonic coordinates require the numerical solution of the Laplace equation with piecewise affine boundary conditions, whereas those that follow need numerical computations (not known analytically) and/or use a partitioning of the polytope.

The construction of generalized barycentric coordinates that satisfy~\eqref{eq:gbc} and are given by a simple analytical expression remains an open problem. In this contribution, we propose an initial attempt to fill this void for common finite element geometries. To this end, we supplement the $d+1$ constraints in~\eqref{eq:gbc} by additional {\em moment} constraints so that
the resulting linear system has a unique nonnegative solution, which is the
desired $\gbc$. We refer to the GBCs so constructed as either {\em moment coordinates (MC)} or {\em moment
shape functions}.
Herein, we develop these coordinates for finite element shapes: quadrilaterals in $\Re^2$ and hexahedra in $\Re^3$.
These geometries can be simple quadrilaterals (convex or nonconvex)
and convex hexahedra.
This idea of adding moment equality conditions to~\eqref{eq:gbc} traces back to sliding mode control in piecewise smooth dynamical systems (see, for completeness,~\cite{dd2:2017} for the quadrilateral in $\R^2$
and~\cite{dd4:2017} for its extension to the hexahedron in $\R^3$), when one looks for Filippov solutions in presence of dynamics along discontinuity manifolds of codimension 2 or higher. The connections of the moment approach
to mean value coordinates are discussed in~\cite{dd2:2017}.

The structure of the remainder of this paper follows. In Section~\ref{sec:formulation},
the formulation to derive moment coordinates is presented: quadrilaterals are
treated in Sections~\ref{subsec:quad} and~\ref{subsec:convQuad},
and piecewise affine interpolation in the
one-dimensional setting is considered in~\sref{subsec:1D}. Extension of moment
coordinates to hexahedra is presented is~\sref{subsec:hex}.
Numerical results in Section~\ref{sec:results} include closed-form expressions and
plots of the moment coordinates on finite element geometries.
We summarize our main finding with some
perspectives on future work
in Section~\ref{sec:conclusions}.

\section{Formulation}\label{sec:formulation}
We first describe the approach to obtain analytical (using symbolic computations) expressions of the
moment coordinates on simple quadrilaterals in~\sref{subsec:quad}
(recover mean value coordinates)
and~\sref{subsec:convQuad} (recover Wachspress coordinates). Then we
consider one-dimensional piecewise affine interpolation in~\sref{subsec:1D}
and moment coordinates on convex hexahedra in~\sref{subsec:hex}.
Our main results provide invertibility for matrices with prescribed sign pattern; however, such matrices elude classical tools for proving invertibility
(for example, see~\cite{KleeLadnerManber:1984:SR,BrualdiShader:1995:MSSLS}) and would
require specific analysis.
We also point out that the sign pattern
is identical to the signs of the hourglass nodal vectors (spurious zero-energy modes
of the stiffness matrix for the Poisson equation) for the four-node quadrilateral and
the eight-node hexahedral element~\cite{FlanaganBelytschko:1981:IJNME}.

\subsection{Four-node (convex or nonconvex) quadrilateral}\label{subsec:quad}

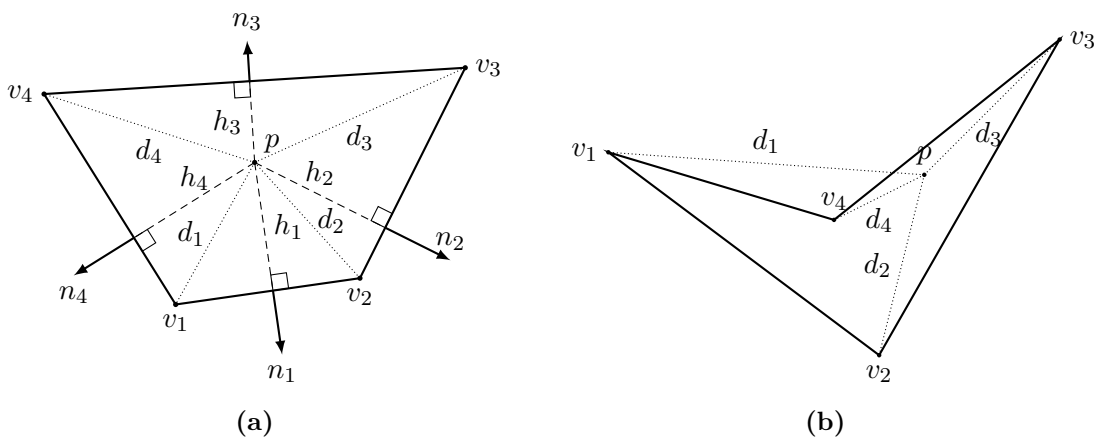
\begin{figure}
\centering
\begin{subfigure}[b]{0.48\textwidth}
\centering
\begin{tikzpicture}[scale=0.7]\usetikzlibrary{calc}
\coordinate [label=below:$v_1$] (A) at (-2.5,-2.5);
% Points from A
\coordinate [label=below:$v_2$]
(B) at (1,-2);
\coordinate [label=right:$v_3$]
(C) at (3,2);
\coordinate [label=left:$v_4$]
(D) at (-5,1.5);
\coordinate [label=above right:$p$]
(P) at (-1,0.2);
\fill (A) circle (1.4pt) (B) circle (1.4pt) (C) circle (1.4pt) (D) circle (1.4pt) (P) circle (1.4pt);
\draw [thick] (A)--(B)--(C)--(D)--cycle;
\draw [densely dotted] (P)--(A) node [pos=0.5, left] {$d_1$};
\draw [densely dotted] (P)--(B) node [pos=0.5, right] {$d_2$};
\draw [densely dotted] (P)--(C) node [pos=0.5, below] {$d_3$};
\draw [densely dotted] (P)--(D) node [pos=0.5, below]
{$d_4$};
%heights
\coordinate (P12) at ($(A)!(P)!(B)$);
\draw [densely dashed] (P)--(P12) node [pos=0.5, right] {$h_1$};
\tkzMarkRightAngle[draw=black,size=.3](P,P12,B);

\coordinate (P23) at ($(B)!(P)!(C)$);
\draw [densely dashed] (P)--(P23) node [pos=0.5, above] {$h_2$};
\tkzMarkRightAngle[draw=black,size=.3](P,P23,C);

\coordinate (P34) at ($(C)!(P)!(D)$);
\draw [densely dashed] (P)--(P34) node [pos=0.5, left] {$h_3$};
\tkzMarkRightAngle[draw=black,size=.3](P,P34,D);

\coordinate (P41) at ($(D)!(P)!(A)$);
\draw [densely dashed] (P)--(P41) node [pos=0.5, above] {$h_4$};
\tkzMarkRightAngle[draw=black,size=.3](P,P41,A);

%normals
\coordinate (P12H) at ($(P12)!-.5!(P)$);
\draw[-latex,thick] (P12)--(P12H) node[below] {$n_1$};

\coordinate (P23H) at ($(P23)!-.5!(P)$);
\draw[-latex,thick] (P23)--(P23H) node[above] {$n_2$};

\coordinate (P34H) at ($(P34)!-.5!(P)$);
\draw[-latex,thick] (P34)--(P34H) node[above] {$n_3$};

\coordinate (P41H) at ($(P41)!-.5!(P)$);
\draw[-latex,thick] (P41)--(P41H) node[below] {$n_4$};
\end{tikzpicture}
% \caption{General convex quadrilateral.}
\caption{}
\label{fig:convQuad}
\end{subfigure}
\hfill
\begin{subfigure}[b]{0.48\textwidth}
\centering
\begin{tikzpicture}[scale=0.6]\usetikzlibrary{calc}
\coordinate [label=left:$v_1$] (A) at (-7,-1.5);
% Points from A
\coordinate [label=below:$v_2$]
(B) at (-1,-6);
\coordinate [label=right:$v_3$]
(C) at (3,1);
\coordinate [label=above:$v_4$]
(D) at (-2,-3);
\coordinate [label=above:$p$]
(P) at (0,-2);
\fill (A) circle (1.4pt) (B) circle (1.4pt) (C) circle (1.4pt) (D) circle (1.4pt) (P) circle (1.4pt);
\draw [thick] (A)--(B)--(C)--(D)--cycle;
\draw [densely dotted] (P)--(A) node [pos=0.5, above] {$d_1$};
\draw [densely dotted] (P)--(B) node [pos=0.5, left] {$d_2$};
\draw [densely dotted] (P)--(C) node [pos=0.3, right] {$d_3$};
\draw [densely dotted] (P)--(D) node [pos=0.5, below]
{$d_4$};
\end{tikzpicture}
\caption{}
% \caption{General nonconvex quadrilateral.}
\label{fig:nonconvQuad}
\end{subfigure}
\caption{Simple quadrilaterals. (a) Convex  and (b) Nonconvex.
$\{v_i\}_{i=1}^4$ are the vertices of the quadrilateral $Q$ and $p$ is a generic point in $Q$. Vertex ordering is
assumed to be cyclic, so that $v_{5}=v_1$ and $v_0=v_4$.}
\label{fig:quad}
\end{figure}

Let $\{v_i\}_{i=1}^{4}\subseteq\R^2$ be affinely independent vertices of a simple quadrilateral (see Figure~\ref{fig:quad}).
The set of all possible barycentric coordinates for any $p\in Q$, $Q\udef\conv\{v_i\}$, is the set of all the nonnegative solutions $\gbc(p)$ to the full-rank, underdetermined system
\begin{equation}\label{eq:BCSystem}
\bmat{\1^\top \\ V}\gbc(p)=\bmat{1 \\ p}.
\end{equation}
As a consequence of the affine independence, such a system has a one-dimensional kernel, so that there exists a nonzero $g\in\R^4$ and $\langle g\rangle=\ker\bmat{\1^\top \\ V}$.
It is straightforward to see that $g$ has a specific sign pattern, which can be computed by
\[
g=\bmat{\tau(v_4) \\ -1},
\]
where $\tau(v_4)$ represents the triangular barycentric coordinate
vector of $v_4$ with respect to the triangle $v_1,v_2,v_3$. Thus, it turns out that
\[
\sgn(g)=\bmat{+ \\ - \\ + \\ -}.
\]
Therefore, if we want to select a specific set of barycentric coordinates for each $p\in Q$, one approach
is to suitably append an additional
condition to~\eqref{eq:BCSystem} of the form
\[
d(p)^\top\gbc(p)=0,\quad d(p)\in\R^4,
\]
so as to make it nonsingular and preserve convexity for its unique solution. Of course, if the vector $d(p)$ is not orthogonal to $g$ then the resulting system would be nonsingular. The most natural way to accomplish this is to select $d(p)$ such that $\sgn(d)=\sgn(g)$. Now, for convexity to be preserved, we would at least guarantee it holds on $\partial Q$. To this end,  let $p\in\overline{v_1v_2}$, so that we expect
\[
\gbc(p)=\bmat{1-\alpha \\ \alpha \\ 0  \\ 0},\quad \alpha\in[0,1],
\]
to be the unique solution to
\begin{equation}\label{eq:momSystem}
\bmat{\1^\top \\ V \\ d(p)^\top}\gbc(p)=\bmat{1 \\ p \\ 0}.
\end{equation}
But since $\alpha=\frac{\|p-v_1\|}{\|v_2-v_1\|}$, we get
\[
\frac{\|p-v_2\|}{\|v_2-v_1\|}d_1+\frac{\|p-v_1\|}{\|v_2-v_1\|}d_2=0,
\]
from which, if $d_1\udef\|p-v_1\|,d_2\udef-\|p-v_2\|$, the above holds true
(see Figures~\ref{fig:convQuad} and \ref{fig:nonconvQuad}).
Repeating the same argument for the remaining edges, we see that defining
\begin{equation}\label{eq:d}
d(p)\udef\bmat{d_1(p) & -d_2(p) & d_3(p) & -d_4(p)},\quad d_i(p)\udef\|p-v_i\|,\,\,i=1,2,3,4,
\end{equation}
provides~\eqref{eq:momSystem} to be nonsingular and its unique solution to be feasible on $\partial Q$. Proving that it is also feasible in the interior of $Q$ can be found in~\cite{dd2:2017}. An alternative proof for the convex quadrilateral case follows.
\begin{prop}\label{prop:momGamma}
For any $p\in Q$ there exist $\alpha,\beta,\gamma\in[0,1]$ such that
\[
p=(1-\gamma)a+\gamma b,\quad
a\udef (1-\alpha)v_4+\alpha v_1,\,\,b\udef (1-\beta)v_2+\beta v_3
\]
and
\begin{equation}\label{eq:momCond}
d(p)^{\top}\gbc(p)=0,\quad\gbc(p)\udef\bmat{(1-\gamma)\alpha & \gamma(1-\beta) & \gamma\beta & (1-\gamma)(1-\alpha)}^{\top},
\end{equation}
where $d(p)$ is as in~\eqref{eq:d}.
\end{prop}
\begin{proof}
Let us assume that $p=0$, and without loss of generality, $d_i=1$ for all $i=1,2,3,4$. Therefore,~\eqref{eq:momCond} reads
\[
\alpha-\gamma\alpha+\gamma\beta=\frac12.
\]
The function $f(\alpha,\beta,\gamma)\udef\alpha-\gamma\alpha+\gamma\beta$ is such that $\nabla f(\alpha,\beta,\gamma)=\bmat{1-\gamma & \gamma & \beta-\alpha}$, so that $\nabla f\neq0$ on $\R^3$. Moreover, on the boundary of $[0,1]^3$ it attains % gets
the following values:
\begin{align*}
f(0,\beta,\gamma) &= \gamma\beta, \\
f(1,\beta,\gamma) &= 1-\gamma(1-\beta), \\
f(\alpha,0,\gamma) &= \alpha(1-\gamma), \\
f(\alpha,1,\gamma) &= \alpha+\gamma(1-\alpha), \\
f(\alpha,\beta,0) &= \alpha, \\
f(\alpha,\beta,1) &= \beta,
\end{align*}
and so has minimum $0$ and maximum $1$. By the Intermediate Value Theorem there exist $\alpha,\beta,\gamma\in[0,1]$ such that $f(\alpha,\beta,\gamma)=\frac12$, which proves the claim.
\end{proof}

Even though moment coordinates are identical to mean value coordinates on a quadrilateral, solving the linear system~\eqref{eq:momSystem} is computationally attractive to obtain $\gbc(p)$ for any point $p \in Q$. This is so, since the local form of computing the mean value coordinates cannot be used if $p \in \partial Q$~\cite{Floater:2015:GBC}. \\

An alternative expression for moment barycentric coordinates on quadrilaterals can be given in terms of triangular barycentric coordinates, which can be seen as a trivial extension of barycentric coordinates relative to one of the triangles when a quadrilateral is split by its diagonals, with a zero component in the position of the removed vertex. For example, leaving $v_1$ out and considering the triangle $v_2,v_3,v_4$, then the corresponding triangular barycentric coordinates $\tau^1(p)$ are computed for each $p\in Q$ as
\[
\bmat{1 & 1 & 1 & 1 \\ v_1 & v_2 & v_3 & v_4 \\ 1 & 0 & 0 & 0}
\tau^1(p)=\bmat{1 \\ p \\ 0}.
\]
Applying Cramer's rule to~\eqref{eq:momSystem} for computing $\gbc_1(p)$ provides
\[
\gbc_1(p)=\frac{\det\bmat{1 & 1 & 1 & 1 \\ p & v_2 & v_3 & v_4 \\ 0 & -d_2 & d_3 & -d_4}}{d(p)^\top\mathbf{n}},
\]
where $\mathbf{n}\udef\bmat{A_{1} \\ -A_{2} \\ A_{3} \\ -A_{4}}$, $A_{i}\udef\abs{\mathrm{area}(\conv\{v_j\,:\,j\neq i\})}$, and $\langle\mathbf{n}\rangle=\ker\bmat{\1^\top \\ V}$. Letting $\hat\tau^1(p)\udef\bmat{\tau^1_2(p) \\ \tau^1_3(p) \\ \tau^1_4(p)}\in\R^3$ and observing that $\mathbf{n}^1(p)\udef A_1\bmat{-1 \\ \hat\tau^1(p)}$ is such that
\[
\langle\mathbf{n}^1(p)\rangle=\ker\bmat{1 & 1 & 1 & 1 \\ p & v_2 & v_3 & v_4},
\]
we have
\begin{align*}
\gbc_1(p) &= \frac{\bmat{0 & -d_2 & d_3 & -d_4}\mathbf{n}^1(p)}{d(p)^\top\mathbf{n}} = A_{1}\frac{\bmat{0 & -d_2 & d_3 & -d_4}\bmat{-1 \\ \hat\tau^1(p)}}{d(p)^\top\mathbf{n}} \\
&= A_{1}\frac{\bmat{d_1 & -d_2 & d_3 & -d_4}\bmat{0 \\ \hat\tau^1(p)}}{d(p)^\top\mathbf{n}} = A_{1}\frac{d(p)^\top\tau^1(p)}{d(p)^\top\mathbf{n}}.
\end{align*}

Analogous arguments hold for the remaining components. Thus, in general, the following holds:
\[
\gbc_i(p)=(-1)^{i+1}A_{i}\frac{d(p)^\top\tau^i(p)}{d(p)^\top\mathbf{n}}.
\]

\subsection{Convex quadrilaterals}\label{subsec:convQuad}

In this section we prove that on a convex quadrilateral $Q$ (see Figure~\ref{fig:convQuad}), akin to mean value coordinates, it is possible to compute Wachspress coordinates by solving a linear system. Let us recall from~\cite{FloaterGilletteSukumar:2014:SINUM} that Wachspress coordinates can be expressed as
\[
\varphi(p) =\frac{w_i(p)}{\sum_{j=1}^{4}w_j(p)},
\]
where
\begin{equation}\label{eq:wWachspress}
w_i(p)=\frac{A(v_{i-1},v_i,v_{i+1})}{A(p,v_{i-1},v_i)A(p,v_i,v_{i+1})},
\end{equation}
and $A(a,b,c)$ stands for the signed area of the triangle with vertices $a,b,c$ given by
\[
A(a,b,c)\udef\frac12\det\begin{bmatrix} a_1 & b_1 & c_1 \\ a_2 & b_2 & c_2 \\ 1 & 1 & 1 \end{bmatrix}.
\]
As is well known (see~\cite{Floater:2015:GBC}), Wachspress coordinates can also, and more conveniently, be expressed in terms of distances from the edges. More precisely, for $i=1,2,3,4$ and with cyclic convention, letting $n_i$ be the outer normal to the edge $\overline{v_iv_{i+1}}$ and $h_i(p)$ be the distance from $p$ to the edge $\overline{v_iv_{i+1}}$, we define the weights
\[
\tilde w_i(p)\udef\frac{n_{i-1}\times n_i}{h_{i-1}(p)h_i(p)},
\]
where, for general $x,y\in\R^2$, we define
\[
x\times y\udef\begin{vmatrix} x_1 & x_2 \\ y_1 & y_2 \end{vmatrix}.
\]
    It then follows that $A(v_{i-1},v_i,v_{i+1})=\frac12 l_{i,i+1}l_{i,i-1}n_{i-1}\times n_i$, where we let $l_{i,i+1}\udef\|v_i-v_{i+1}\|$, for $i=1,2,3,4$. It can be seen that $w_i(p)=2\tilde w_i(p)$. Now, letting
\[
\rho_i(p)\udef l_{i,i-1}l_{i,i+1}h_{i-1}(p)h_i(p),
\]
and using planar geometry, it is readily shown that
\[
\rho_1(p)\varphi_{1}(p)-\rho_2(p)\varphi_{2}(p)+
\rho_3(p)\varphi_{3}(p)-\rho_4(p)\varphi_{4}(p)=0.
\]
Thus, letting
\[
\rho(p)\udef\bmat{\rho_1(p) & -\rho_2(p) & \rho_3(p) & -\rho_4(p)}^\top ,
\]
it follows that $\rho(p)^\top\mathbf{n}\neq0$ because of the sign pattern, and so Wachspress coordinates are the unique solution to the linear system
\begin{equation}\label{eq:wach}
\bmat{V \\ \1^\top \\ \rho(p)^\top}\varphi(p)=\bmat{p \\ 1 \\ 0}.
\end{equation}
% Let us notice that,
We point out that $p\in Q$, $p\in\overline{v_iv_{i+1}}$ if and only if $h_i(p)=0$. Therefore, $p\in\overline{v_{i-1}v_i}\cup\overline{v_iv_{i+1}}$ if and only if $\rho_i(p)=0$.
% Letting $A_{ijk}\udef A(v_i,v_j,v_k)$ for $i<j<k$ cyclically, and $A\udef\mathrm{area}Q$, signed according to the common sign of all the $A_{ijk}$'s, then we have
% \[
% A_{123}=A-A_{341},\quad A_{234}=A-A_{412},
% \]
% from which, subtracting both sides,
% \[
% A_{123}-A_{234}=A_{412}-A_{341}\Rightarrow A_{234}-A_{341}+A_{412}-A_{123}=0.
% \]
% Letting $A_i(p)\udef A(p,v_i,v_{i+1})$ and since from \eqref{eq:wWachspress}
% \[
% A_{i-1,i,i+1}=w_i(p)A_{i-1}(p)A_{i}(p),
% \]
% then we have
% \[
% A_4(p)A_1(p)\lambda_{W,1}-A_1(p)A_2(p)\lambda_{W,2}+A_2(p)A_3(p)\lambda_{W,3}-A_3(p)A_4(p)\lambda_{W,4}=0.
% \]
% Thus, letting
% \[
% a(p)\udef\bmat{A_4(p)A_1(p) & -A_1(p)A_2(p) & A_2(p)A_3(p)& -A_3(p)A_4(p)}^\top,
% \]
% it follows $a(p)^\top\mathbf{n}\neq0$ because of the sign pattern, and so Wachspress barycentric coordinates are the unique solution to the linear system
% \[
% \bmat{V \\ \1^\top \\ a(p)^\top}\lambda_W(p)=\bmat{p \\ 1 \\ 0}.
% \]

\subsection{One-dimensional case}\label{subsec:1D}
Moment conditions can be leveraged to recover piecewise linear basis functions on an interval in one dimension.
Let us start with the case of three points in $[0,1]$, say $x_1=0,x_2\in(0,1),x_3=1$. Letting $x\in[0,1]$, we then have to consider the underdetermined linear system
\begin{equation}\label{eq:1D_3nodes}
\bmat{1 & 1 & 1 \\ x_1 & x_2 & x_3}\gbc=\bmat{1 \\ x}.
\end{equation}
\begin{figure}
\centering
\begin{tikzpicture}
\draw plot [mark=*] coordinates
{(0,0) (5,0) (10,0)};
\draw plot [mark=x] coordinates {(3,0)};
\node at (3,-0.5) {$p$};
\draw plot [mark=*] coordinates {(0,3)};
\draw plot [mark=*] coordinates {(5,-2)};
\draw[dashed] (0,3)--(5,-2);
\node at (0,0.5) {$x_1$};
\node at (5,0.5) {$x_2$};
\node at (10,0.5) {$x_3$};
\node at (3,0.5) {$x$};
\node at (-0.5,3) {$A$};
\node at (5.5,-2) {$B$};
\end{tikzpicture}
\caption{Three nodes in one dimension.}
\label{fig:1D_3nodes}
\end{figure}
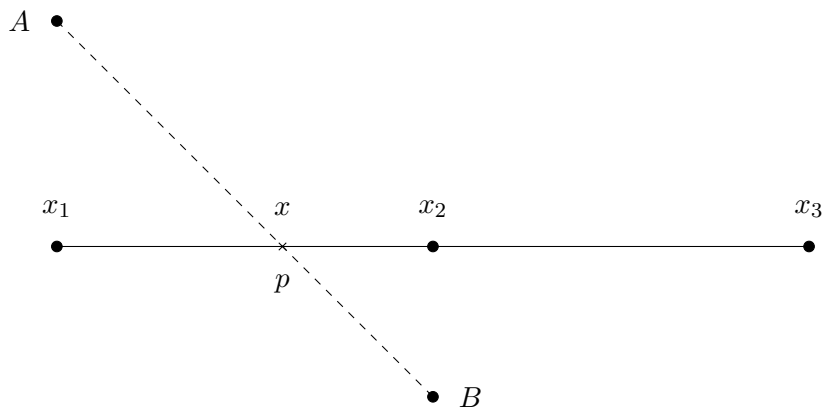
The idea is to regularize the above system so that the point $(x,0)$ belongs to the edge $AB$ (see Figure~\ref{fig:1D_3nodes}), where $A(x_0,x-x_1),B(x_1,x-x_2))$, while ensuring that~\eqref{eq:1D_3nodes} is invertible so that its unique solution is componentwise nonnegative. One way to accomplish this is to append
to~\eqref{eq:1D_3nodes} an additional row using moment regularization. Specifically, we propose to regularize~\eqref{eq:1D_3nodes} as
\[
\bmat{1 & 1 & 1 \\ x_1-x & x_2-x & x_3-x \\ |x_1-x| & -|x_2-x| & |x_3-x|}\gbc=\bmat{1 \\ 0 \\ 0}.
\]
The system above is easily seen to be invertible, and since $\bmat{0\\0}$ is in the convex hull of the points $\bmat{x_1-x\\|x_1-x|},\bmat{x_2-x\\-|x_2-x|},\bmat{x_3-x\\|x_3-x|}$, then its unique solution $\gbc(x)$ is necessarily feasible. \\

In the following we extend such an argument to any number $n$ of distinct points $\{x_i\}_{i=1}^{n}$ in $[0,1]$. We recall a handy result from~\cite{dd4:2017}.
\begin{lem}\label{lemma:Ans}
Let $A\in\R^{n\times m}$, $n<m$, be full rank, and let $b\in\R^{n}$.
Consider the system
\begin{equation}\label{eq:sAns}
Ax=b\ ,
\end{equation}
and let $d\in\R^{m}$ be a nonzero vector.

If there exist $x$ and $y$ solutions to~\eqref{eq:sAns}, such that
\begin{equation*}
d^{\top}x = \xi\ , \quad \text{and}\quad
d^{\top}y = \eta\ ,
\end{equation*}
with $\xi\neq\eta$, then
$\begin{bmatrix} A \\ d^{\top} \end{bmatrix}$ has rank $n+1$.
\end{lem}

Let $n\geq4$, $V\udef\bmat{x_1 & \cdots & x_n}$, with $0\leq x_1<x_2<\ldots<x_n\leq1$. Then, given $p\in[0,1]$, we seek for a feasible solution $\gbc(p)$ of the underdetermined system
\begin{equation}\label{eq:1D_n_nodes}
\bmat{1 & \cdots & 1 & \cdots & 1 \\ x_1 & \cdots & x_i & \cdots & x_n \\ |x_1-x| & \cdots & (-1)^{i+1}|x_i-x| & \cdots & (-1)^{n+1}|x_n-x|}\gbc=\bmat{1 \\ x \\ 0},
\end{equation}
which has full rank $3$.

Without loss of generality, let us assume $p\in[x_1,x_2]$; other cases can be handled analogously.
Let us note that the last row above, which may be referred to as \emph{moment regularization row}, is always compatible with the expected solution to~\eqref{eq:1D_n_nodes}, which has the form
\[
\gbc_i(p)=
\begin{cases}
(1-\alpha), & i=1, \\
\alpha, & i=2, \\
0, & i\neq1,2.
\end{cases}
\]
The case $n=4$ is the classical moment system, which is known to be invertible, providing a unique feasible solution. For $n\geq5$, let us define
\begin{align*}
d_i(p) &\udef (-1)^{i+1}|x_i-p|,\quad i=1,\ldots,n, \\
\Delta &\udef \bmat{
0 & 0 & 1 & 1 & \cdots & 0 \\
\vdots & \vdots & & \ddots & \ddots & \vdots \\
0 & 0 & \cdots & \cdots & 1 & 1}\in\R^{(n-3)\times n}.
\end{align*}
We prove the following.
\begin{prop}\label{prop:mom1D}
Let $p\in[0,1]$. For all $n\geq4$ the system
\[
\bmat{\1^\top \\ V-p \\ d(p)^\top \\ \Delta}\gbc(p)=\bmat{1 \\ 0 \\ 0 \\ 0_{n-3}}
\]
is invertible and its unique solution is feasible.
\end{prop}
\begin{proof}
We proceed by induction on $n$. Observe that the base of induction is the case $n=4$. If so,
\[
\Delta=\bmat{0 & 0 & 1 & 1}.
\]
Assume $p=x_1$, let $\alpha\neq0$ such that
\[
p=(1-\alpha)x_2+\alpha x_3.
\]
Therefore $0=\Delta\bmat{1 \\ 0 \\ 0 \\ 0}\neq\Delta\bmat{0 \\ 1-\alpha \\ \alpha \\ 0}=\alpha$, and Lemma~\ref{lemma:Ans} proves the claim. An analogous argument hold if $p=x_2$. If $p\in(x_1,x_2)$, then let $\alpha,\beta\in(0,1)$ such that
\[
p=(1-\alpha)x_1+\alpha x_2=(1-\beta)x_1+\beta x_4.
\]
Again $0=\Delta\bmat{1-\alpha \\ \alpha \\ 0 \\ 0}\neq\Delta\bmat{1-\beta \\ 0 \\ 0 \\ \beta}=\beta$, and Lemma~\ref{lemma:Ans} proves the claim. \\
Let the result hold true for $n-1$, and let us define
\[
\delta^\top\udef\bmat{0 & 0 & \cdots & 0 & 1 & 1}\in\R^n.
\]
Let $p\in(x_1,x_2)$. Then there exist $\alpha,\beta\in(0,1)$ such that
\[
p=(1-\alpha)x_1+\alpha x_2
\]
and
\[
p=
\begin{cases}
(1-\beta)x_1+\beta x_{n-1}, & \textrm{if $n$ is odd}, \\
(1-\beta)x_1+\beta x_n, & \textrm{if $n$ is even}.
\end{cases}
\]
Therefore, assuming $n$ to be even, $\gbc_{12}\udef\bmat{1-\alpha \\ \alpha \\ 0 \\ \vdots \\ 0}, \
\gbc_{1n}\udef\bmat{1-\beta \\ 0 \\ \vdots \\ 0 \\ \beta}$ are solutions to
\[
\bmat{\1^\top & 1 \\ V_{1:n-1}-p & x_n-p \\ d(p)^\top & -|x_n-x| \\ \Delta & 0}\gbc(p)=\bmat{1 \\ 0 \\ 0 \\ 0_{n-3}},
\]
and $\delta^\top\gbc_{12}=0\neq\beta=\delta^\top\gbc_{1n}$, so that Lemma~\ref{lemma:Ans} proves the result. Similar arguments hold for the boundary cases $p=x_1$ or $p=x_2$. We proceed analogously for $n$ odd. The claim is proved.
\end{proof}

\subsection{Flat-faced hexahedron} \label{subsec:hex}

We begin by first stating
the main result from~\cite{dd4:2017}.

\begin{thm}\label{thm:Minv}
Let $v_i=\smat{v_i^1\\ v_i^2\\ v_i^3}$, $i=1,\dots, 8$, be eight vectors in $\R^3$,
and consider the matrix
$V\in \R^{3\times 8}$ given by
\begin{equation}\label{eq:Vmat}
V\udef \bmat{v_1&v_2&v_3&v_4&v_5&v_6&v_7&v_8}\ .
\end{equation}
Let $p\in H$, $H := \conv\{v_i\}$.
Assume that the entries of $V-p$ are nonzero and have the following signs:
\begin{equation}\label{eq:spV}
\begin{bmatrix}
+ & + & + & + & - & - & - & - \\
+ & + & - & - & + & + & - & - \\
+ & - & - & + & + & - & - & +
\end{bmatrix}.
\end{equation}
Let $\Delta(p)\in \R^{3\times 8}$ be the
following  matrix of {\em signed}
partial distances:
\begin{equation}\label{eq:Delta}
\Delta(p)\udef
\begin{bmatrix}
\delta^{23}_{1}(p) & -\delta^{23}_{2}(p) & \delta^{23}_{3}(p) & -\delta^{23}_{4}(p) & \delta^{23}_{5}(p) & -\delta^{23}_{6}(p) &
\delta^{23}_{7}(p) & -\delta^{23}_{8}(p) \\
\delta^{13}_{1}(p) & -\delta^{13}_{2}(p) & -\delta^{13}_{3}(p) & \delta^{13}_{4}(p) & -\delta^{13}_{5}(p) & \delta^{13}_{6}(p) &
\delta^{13}_{7}(p) & -\delta^{13}_{8}(p) \\
\delta^{12}_{1}(p) & \delta^{12}_{2}(p) & -\delta^{12}_{3}(p) & -\delta^{12}_{4}(p) & -\delta^{12}_{5}(p) & -\delta^{12}_{6}(p) &
\delta^{12}_{7}(p) & \delta^{12}_{8}(p)
\end{bmatrix},
\end{equation}
where for each $i=1,\ldots,8$,
\begin{align*}
\delta^{23}_{i}(p) &\udef \sqrt{(v_{i}^{2}-p^2)^{2}+(v_{i}^{3}-p^3)^{2}}\ , \\
\delta^{13}_{i}(p) &\udef \sqrt{(v_{i}^{1}-p^1)^{2}+(v_{i}^{3}-p^3)^{2}}\ , \\
\delta^{12}_{i}(p) &\udef \sqrt{(v_{i}^{1}-p^1)^{2}+(v_{i}^{2}-p^2)^{2}}\ .
\end{align*}
Finally, let
\[
d(p)^{\top}\udef\begin{bmatrix} d_{1}(p) & -d_{2}(p) & d_{3}(p) & -d_{4}(p) & -d_{5}(p) & d_{6}(p) & -d_{7}(p) & d_{8}(p) \end{bmatrix},
\]
where for each $i=1,\ldots,8$,
\[
d_{i}(p)\udef\|v_{i}-p\|_{2},
\]
and let $\1\in \R^8$ be the vector of all $1$'s.

Then, the system
\begin{equation}\label{eq:Mcod3_}
M(p)\gbc(p)=\bmat{1 \\ p \\ 0_4},\quad M(p)\udef\begin{bmatrix} \1^\top \\ V \\ \Delta(p) \\ d(p)^{\top} \end{bmatrix}
\end{equation}
has a unique, componentwise nonnegative solution $\gbc(p)$.
\end{thm}

However, it can be seen that if $p\in\partial H$ then \eqref{eq:Mcod3_}, while still nonsingular, can provide a unique unfeasible solution $\gbc(p)$. This issue can be overcome by suitably placing zeros on all the columns of $\Delta(p)$ in \eqref{eq:Delta} relative to the indices of the face that $p$ belongs to. In order to do so, letting $\mathcal{F}$ be the planar face with vertices $v_i,v_j,v_k,v_h$, if it holds that
\[
(p-v_i)\times(v_j-v_i)=(v_k-v_i)\times(v_j-v_i),
\]
then $p\in\mathcal{F}$ and thus we set
\[
\Delta(:,[i,j, k, h])=0_{3\times 4}.
\]
It is a consequence from the 2D case that the corresponding $M(p)$ is invertible and that the unique solution $\gbc(p)$ to~\eqref{eq:Mcod3_} is componentwise nonnegative, such that $\gbc_{r}(p)=0$, for $r=i,j,k,h$.

It is fundamental to point out
that~\eqref{eq:spV} is a crucial assumption to obtain feasibility of the unique solution to~\eqref{eq:Mcod3_}. A simple argument is needed to show that~\eqref{eq:spV} is always satisfied, for each $p\in H$, if $H$ is a regular hexahedron (for example, if $H$ is a cube). However, when $H$ is a general convex hexahedron, one can easily see that~\eqref{eq:spV} may fail. In this case, a simple change of coordinates can fix this issue. More precisely, if $H$ is convex, one can always find three linearly independent normals, providing three planes, such that each vertex in $V$ lies in a different orthant, as shown below.
\begin{lem}\label{lem:changeOfCoordinates}
Let $H$ be a convex hexahedron. Then, for each $p\in H$ there exists a unit reference system $R(p)\udef\{r_1(p),r_2(p),r_3(p)\}$ such that the sign pattern of $V-p$, in $R(p)$ is given by~\eqref{eq:spV}.
\end{lem}
\begin{proof}
Let $F_i, G_i, H_i, i=1,2$ be the three couples of opposite faces of $H$. If $F_1, F_2$ are parallel, then let $r_1$ be any vector parallel to $F_1$. Otherwise, if $F_1, F_2$ are not parallel, let $l\udef S_{F_1}\cap S_{F_2}$ be the portion of line at the intersection of the supporting planes of $F_1, F_2$ and within the polyhedron defined by the supporting planes of all the other faces. Let $A\in l$ be arbitrarily chosen. Then, let $r_1$ be the unit vector relative to $p-A$. The remaining vectors $r_2, r_3$ are constructed analogously as above, according to whether $G_i, H_i, i=1,2$ are parallel or not. The set $\{r_1,r_2,r_3\}$ defined this way is then a basis for $\R^3$ such that $\sgn(V-p)$ is as in~\eqref{eq:spV}.
\end{proof}

In addition, we observe that if $H$ is a nonconvex hexahedron, the conclusion of Lemma \ref{lem:changeOfCoordinates} may fail depending on the point $p\in H$.

\section{Numerical results}\label{sec:results}

We present analytical expressions of moment coordinates over
convex and nonconvex quadrilaterals, and present plots of these coordinates over
quadrilaterals and hexahedra.

\subsection{Quadrilateral}\label{subsec:results_quad}
On solving~\eqref{eq:momSystem} over the biunit square ($Q = \square$) that is centered at $[1/2,1/2]^\top$,
we obtain the analytical solution for the moment coordinates (identical to mean value coordinates):
\begin{equation*}
\resizebox{\textwidth}{!}{$
\gbc =
    \left[\begin{array}{c} -\frac{x \sqrt{x^2-2x+y^2-2y+2}-\sqrt{x^2+2x+y^2-2y+2}-\sqrt{x^2-2x+y^2+2y+2}+x\sqrt{x^2+2x+y^2-2y+2}+y\sqrt{x^2-2x+y^2-2y+2}+y\sqrt{x^2-2x+y^2+2y+2}}
    {2\left(\sqrt{x^2-2x+y^2-2y+2}+\sqrt{x^2-2x+y^2+2y+2}+\sqrt{x^2+2x+y^2-2y+2}+\sqrt{x^2+2x+y^2+2y+2}\right)} \\
    \frac{\sqrt{x^2-2x+y^2-2y+2}+\sqrt{x^2+2x+y^2+2y+2}+x\sqrt{x^2-2x+y^2-2y+2}+x\sqrt{x^2+2x+y^2-2y+2}-y\sqrt{x^2+2x+y^2-2y+2}-y\sqrt{x^2+2x+y^2+2y+2}}
    {2\left(\sqrt{x^2-2x+y^2-2y+2}+\sqrt{x^2-2x+y^2+2y+2}+\sqrt{x^2+2x+y^2-2y+2}+\sqrt{x^2+2x+y^2+2y+2}\right)} \\
    \frac{\sqrt{x^2-2\,x+y^2+2y+2}+\sqrt{x^2+2x+y^2-2\,y+2}+x\sqrt{x^2-2x+y^2+2y+2}+x\sqrt{x^2+2x+y^2+2y+2}+y\sqrt{x^2+2x+y^2-2y+2}+y\sqrt{x^2+2x+y^2+2y+2}}
    {2\left(\sqrt{x^2-2x+y^2-2y+2}+\sqrt{x^2-2x+y^2+2y+2}+\sqrt{x^2+2x+y^2-2y+2}+\sqrt{x^2+2x+y^2+2y+2}\right)} \\
    \frac{\sqrt{x^2-2x+y^2-2y+2}+\sqrt{x^2+2x+y^2+2y+2}-x\sqrt{x^2-2x+y^2+2y+2}-x\sqrt{x^2+2x+y^2+2y+2}+y\sqrt{x^2-2x+y^2-2y+2}+y\sqrt{x^2-2x+y^2+2y+2}}
    {2\left(\sqrt{x^2-2x+y^2-2y+2}+\sqrt{x^2-2x+y^2+2y+2}+\sqrt{x^2+2x+y^2-2y+2}+\sqrt{x^2+2x+y^2+2y+2}\right)}
    \end{array}\right].
$
}
\end{equation*}
As is well known, mean value coordinates are distinct from bilinear finite element shape functions on the biunit square.  However, Wachspress coordinates are identical to bilinear finite element
shape functions which is realized on solving~\eqref{eq:wach}.
We now present these coordinates on two simple (convex and nonconvex) quadrilaterals.

\begin{exm}\label{ex:conv}
Let us consider the convex quadrilateral $Q$ with vertices
\[
v_1=\bmat{0 \\ 0}, \ v_2=\bmat{1 \\ 0}, \ v_3=\bmat{\frac12 \\ 4}, \ v_4=\bmat{0 \\ 2}.
\]
\end{exm}
Figure~\ref{fig:exConv_mc_wc} shows both mean value and Wachspress coordinates for each vertex, while their partial derivatives are shown in Figure~\ref{fig:exConv_mc_der} and Figure~\ref{fig:exConv_wc_der}, respectively. Both coordinates are nonnegative, piecewise affine on the boundary and satisfy the Kronecker-delta property at the vertices. Furthermore, we
obtain the analytical solution for Wachspress coordinates:
\[
\gbc =
\left[\begin{array}{c}
\frac{-32\,x^2+4\,x\,y+16\,x+y^2-10\,y+16}{2\,\left(16\,x-y+8\right)}\\ \frac{4\,x\,\left(4\,x-y+2\right)}{16\,x-y+8}\\ \frac{6\,x\,y}{16\,x-y+8}\\ \frac{y\,\left(8-8\,x-y\right)}{2(16\,x-y+8)}
\end{array}\right]
,
\]
which are rational functions, with numerator of degree $2$ and denominator of degree $1$~\cite{Wachspress:2016:RBG}.
\begin{figure}
    \centering
\begin{subfigure}[b]{\textwidth}
    \begin{subfigure}[b]{0.25\textwidth}
    \includegraphics[width=\textwidth]{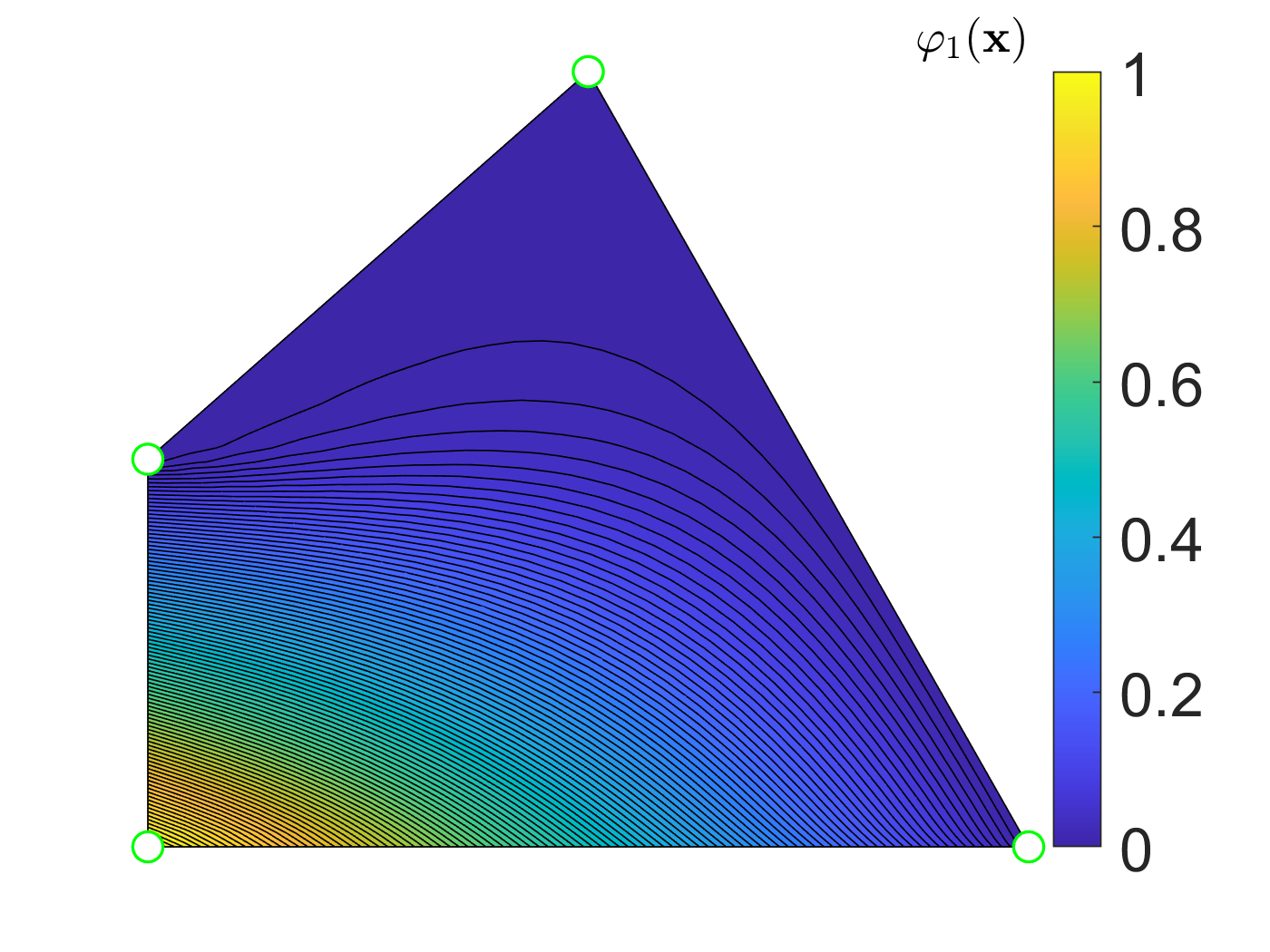}
    \end{subfigure}%
    \hfill
    \begin{subfigure}[b]{0.25\textwidth}
    \includegraphics[width=\textwidth]{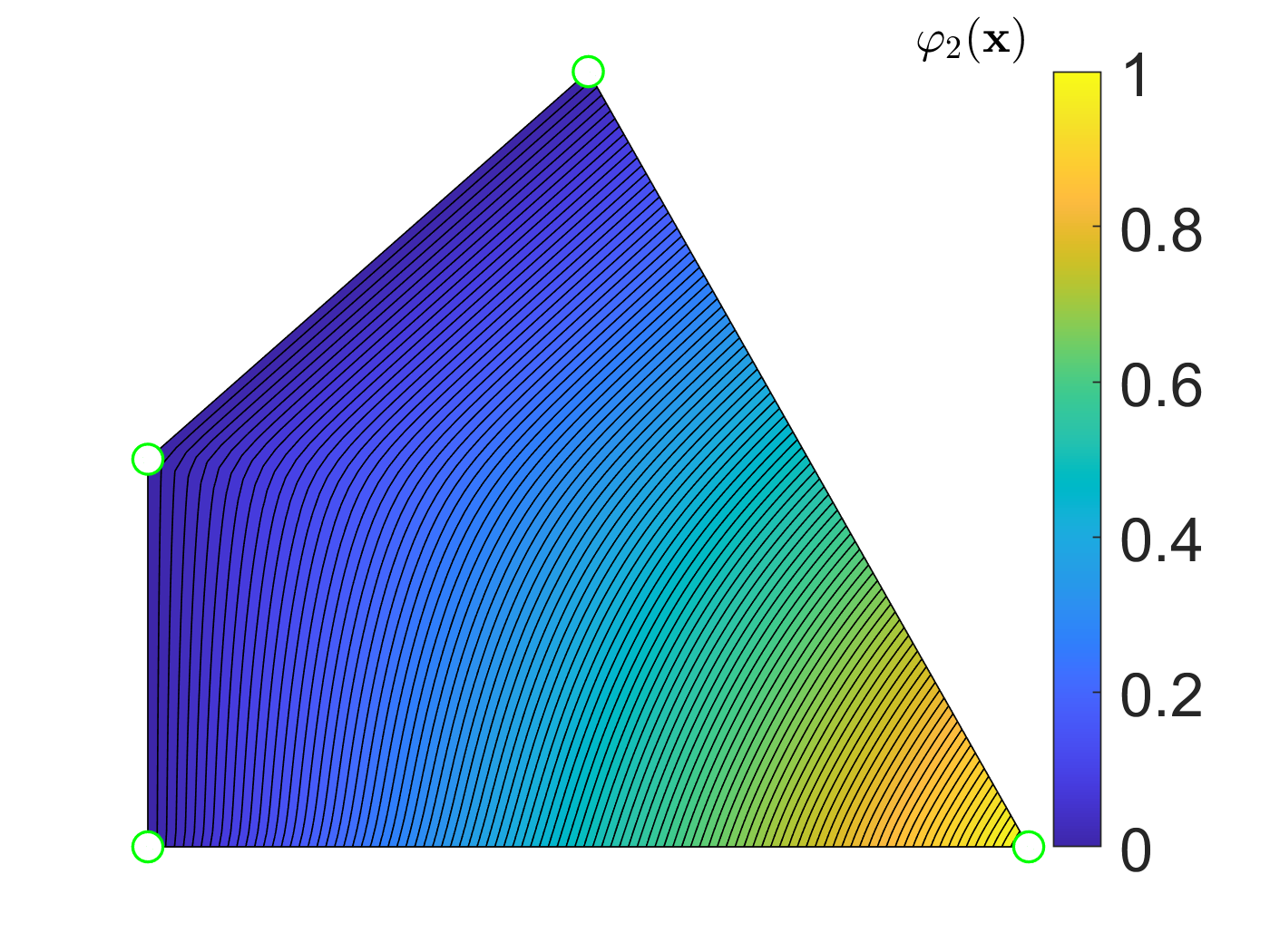}
    \end{subfigure}%
    \hfill
    \begin{subfigure}[b]{0.25\textwidth}
    \includegraphics[width=\textwidth]{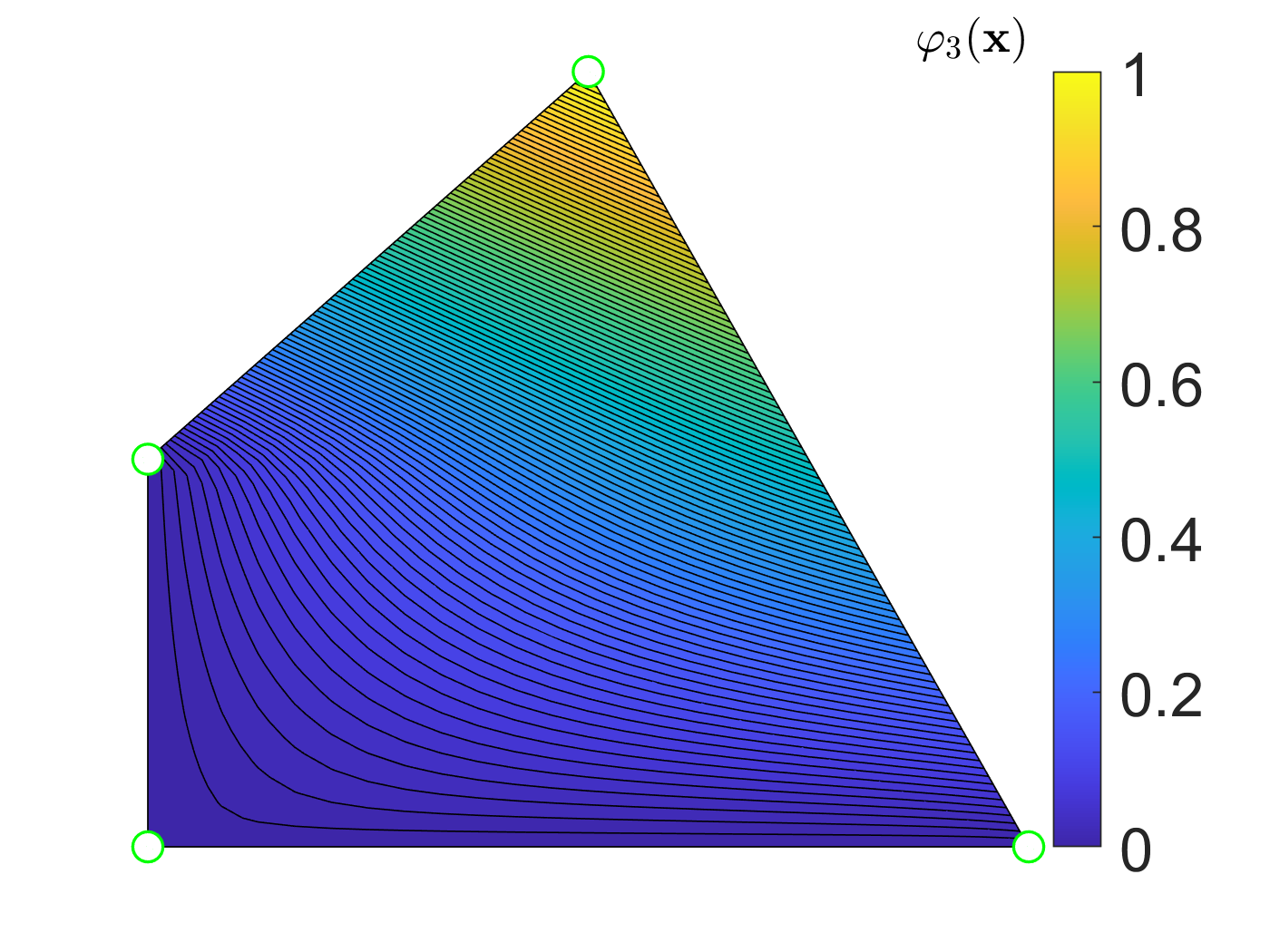}
    \end{subfigure}%
    \hfill
    \begin{subfigure}[b]{0.25\textwidth}
    \includegraphics[width=\textwidth]{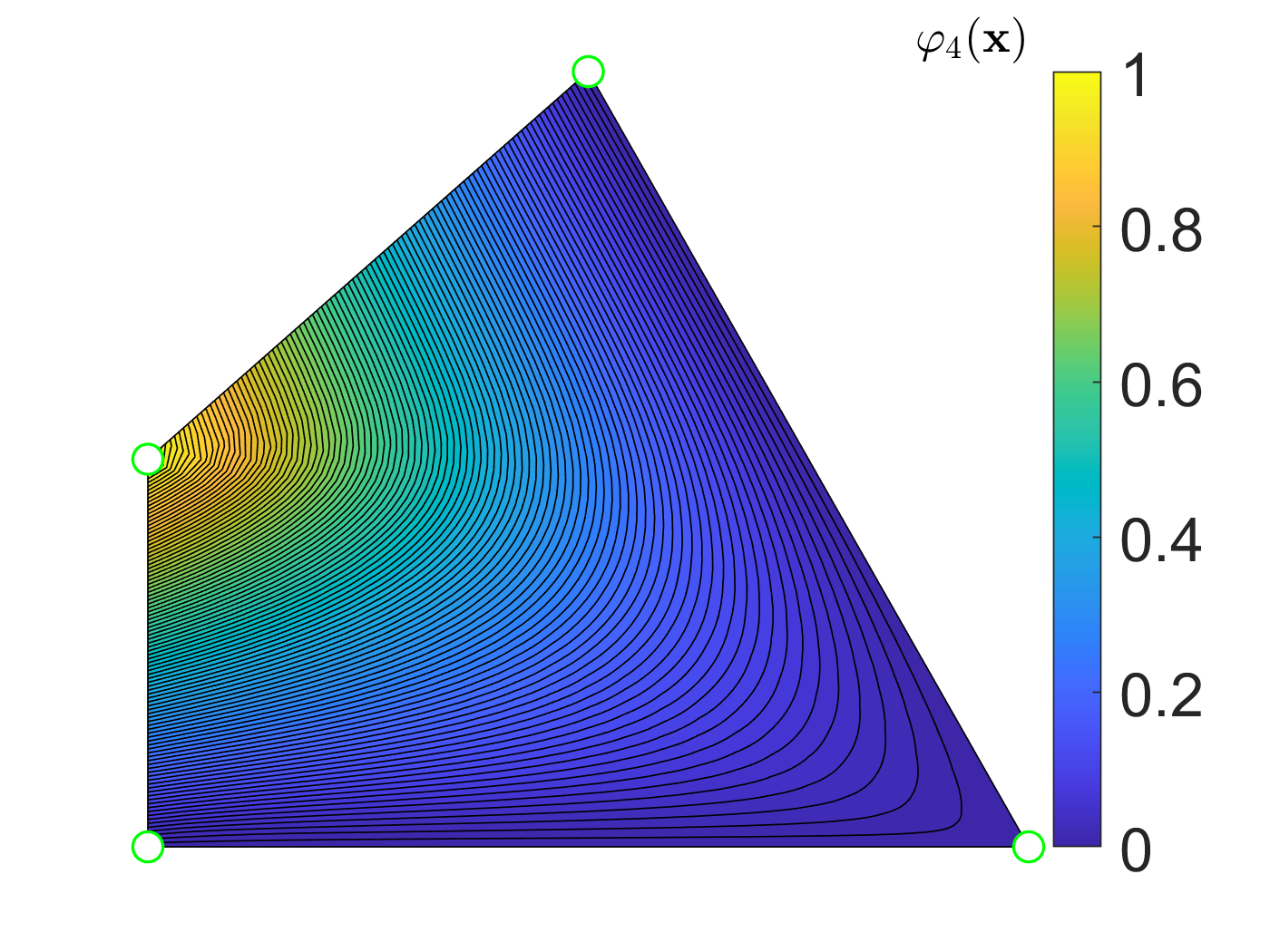}
    \end{subfigure}
    \caption{}
    \label{fig:exConv_mc}
\end{subfigure}
\begin{subfigure}[b]{\textwidth}
    \begin{subfigure}[b]{0.25\textwidth}
    \includegraphics[width=\textwidth]{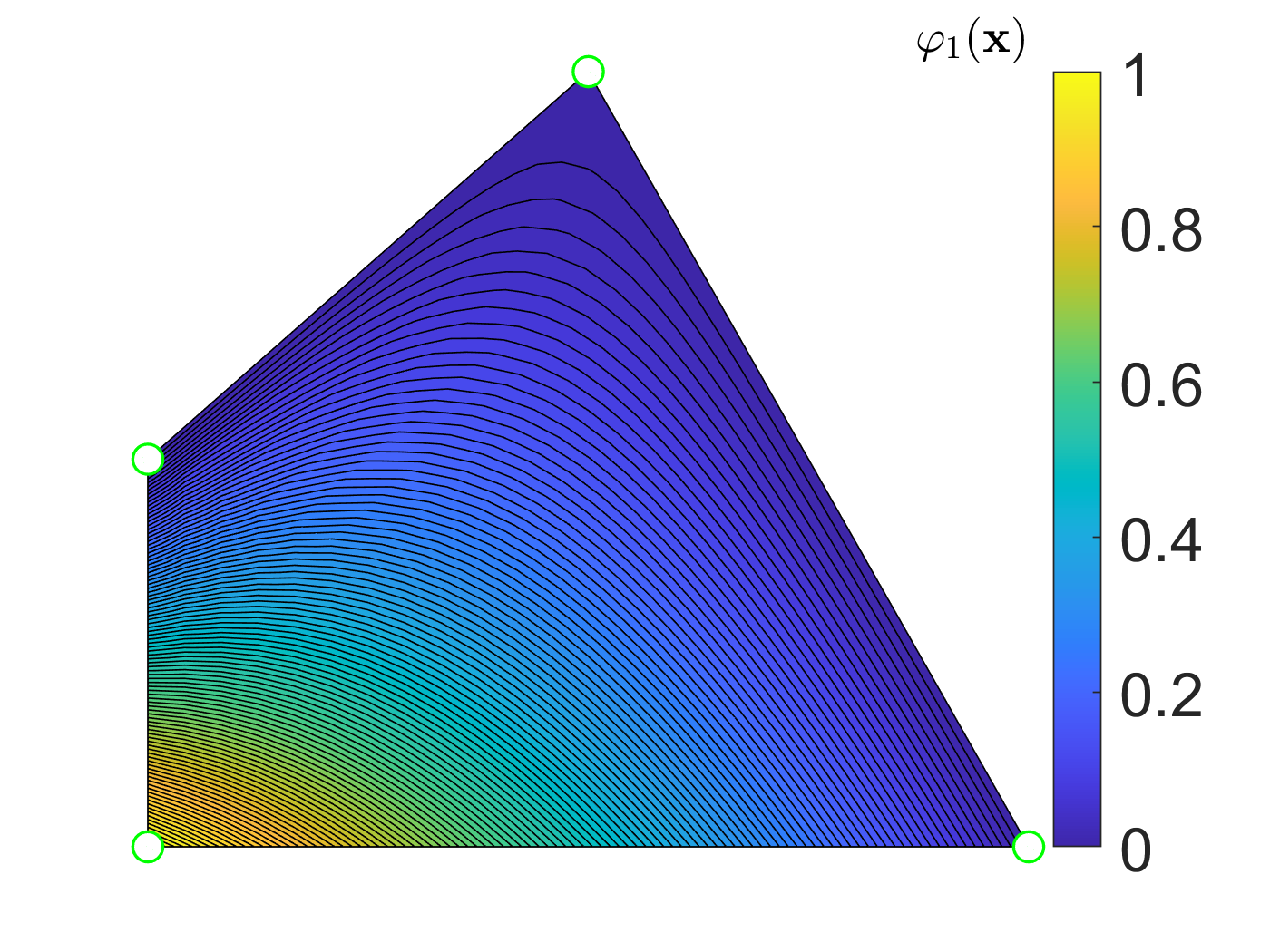}
    \end{subfigure}%
    \hfill
    \begin{subfigure}[b]{0.25\textwidth}
    \includegraphics[width=\textwidth]{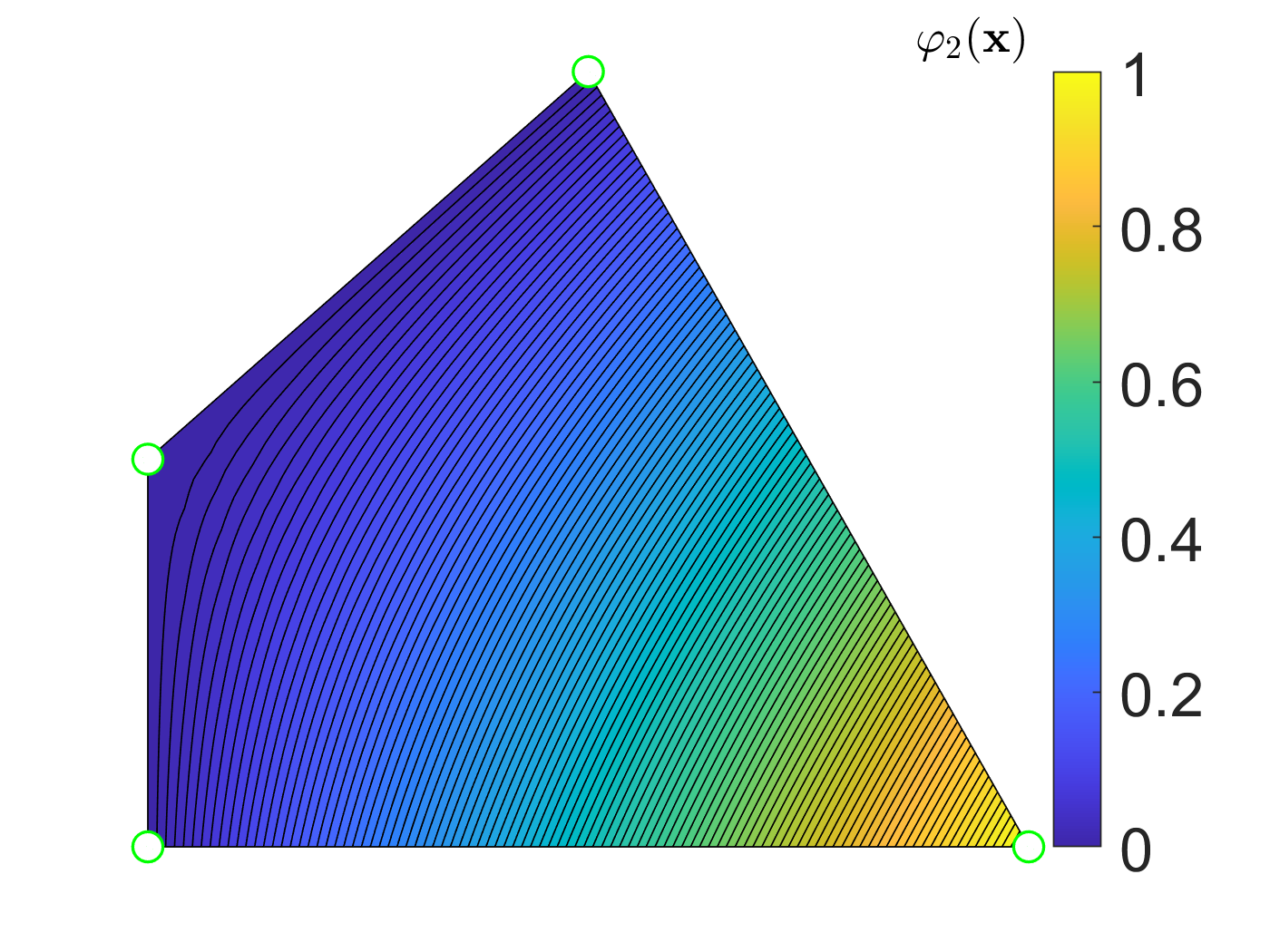}
    \end{subfigure}%
    \hfill
    \begin{subfigure}[b]{0.25\textwidth}
    \includegraphics[width=\textwidth]{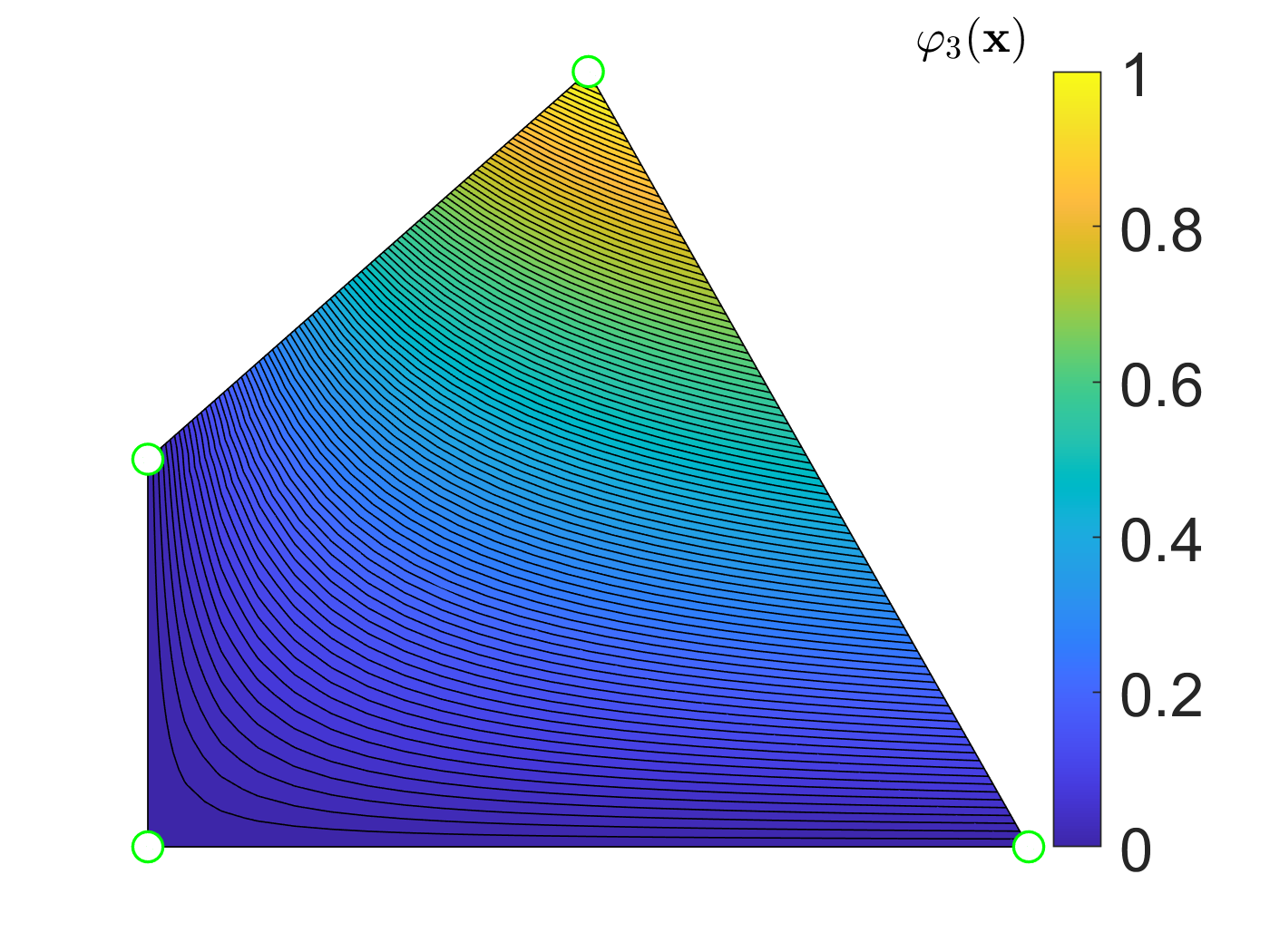}
    \end{subfigure}%
    \hfill
    \begin{subfigure}[b]{0.25\textwidth}
    \includegraphics[width=\textwidth]{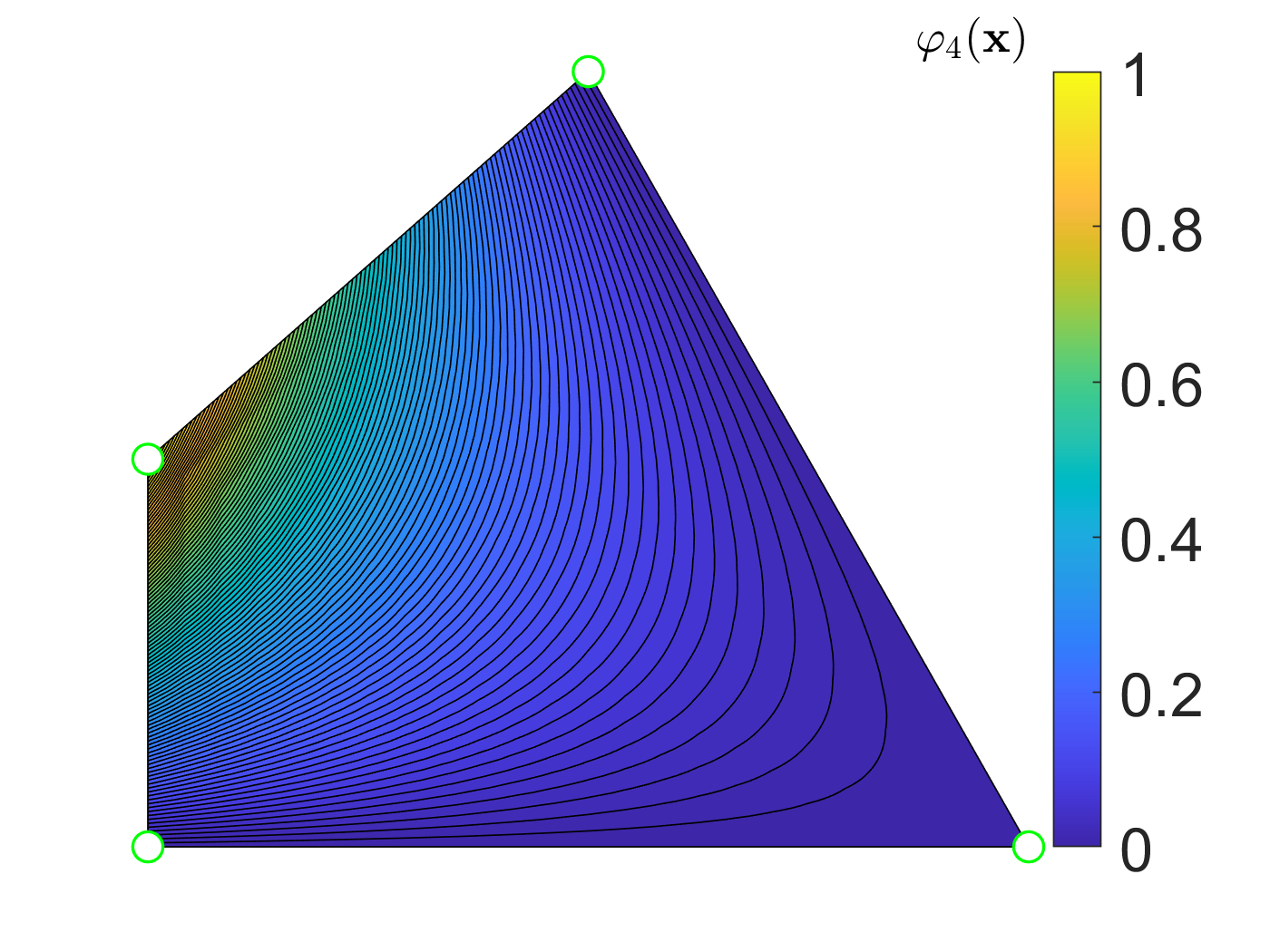}
    \end{subfigure}
    \caption{}
    \label{fig:exConv_wc}
\end{subfigure}

    \caption{Plots of (a) moment and (b) Wachspress
    coordinates on $Q$ for Example~\ref{ex:conv}.}
    \label{fig:exConv_mc_wc}
\end{figure}

\begin{figure}
    \centering
\begin{subfigure}[b]{\textwidth}
    \begin{subfigure}[b]{0.25\textwidth}
    \includegraphics[width=\textwidth]{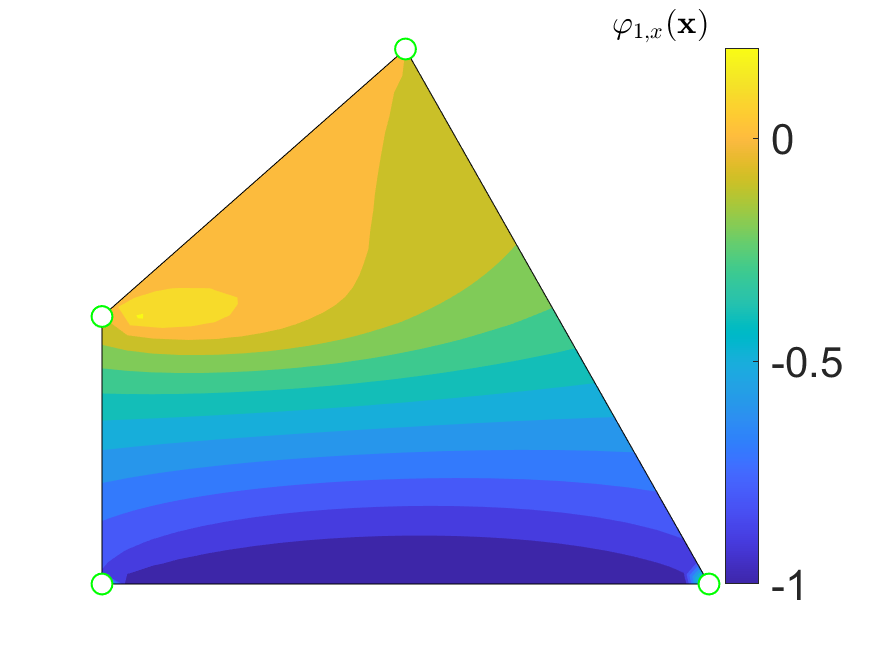}
    \end{subfigure}%
    \hfill
    \begin{subfigure}[b]{0.25\textwidth}
    \includegraphics[width=\textwidth]{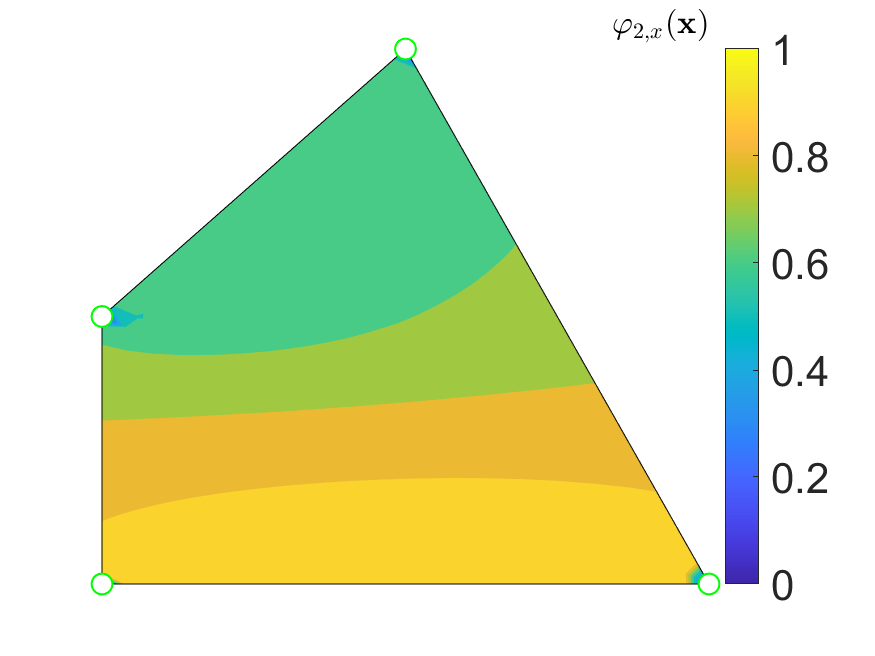}
    \end{subfigure}%
    \hfill
    \begin{subfigure}[b]{0.25\textwidth}
    \includegraphics[width=\textwidth]{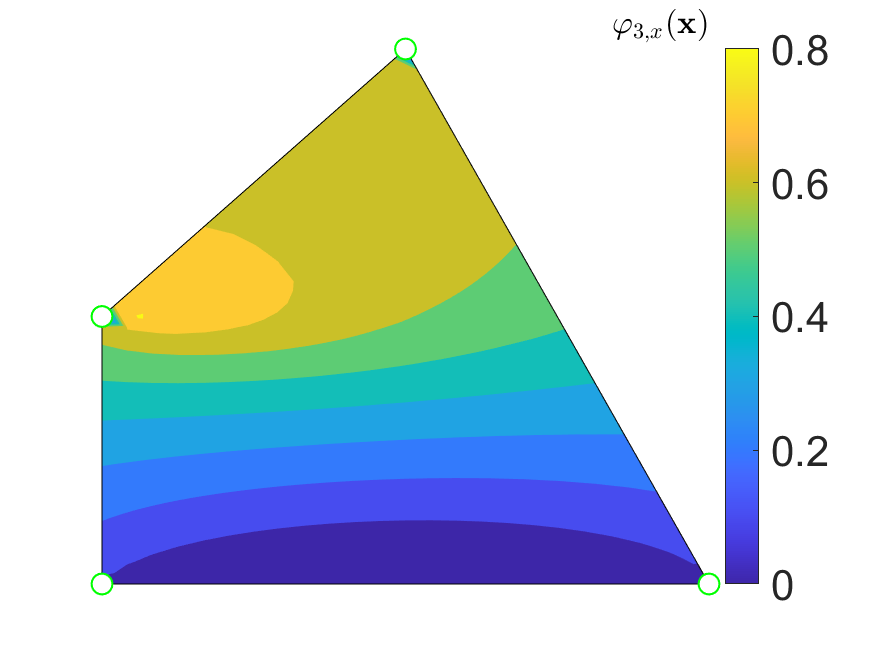}
    \end{subfigure}%
    \hfill
    \begin{subfigure}[b]{0.25\textwidth}
    \includegraphics[width=\textwidth]{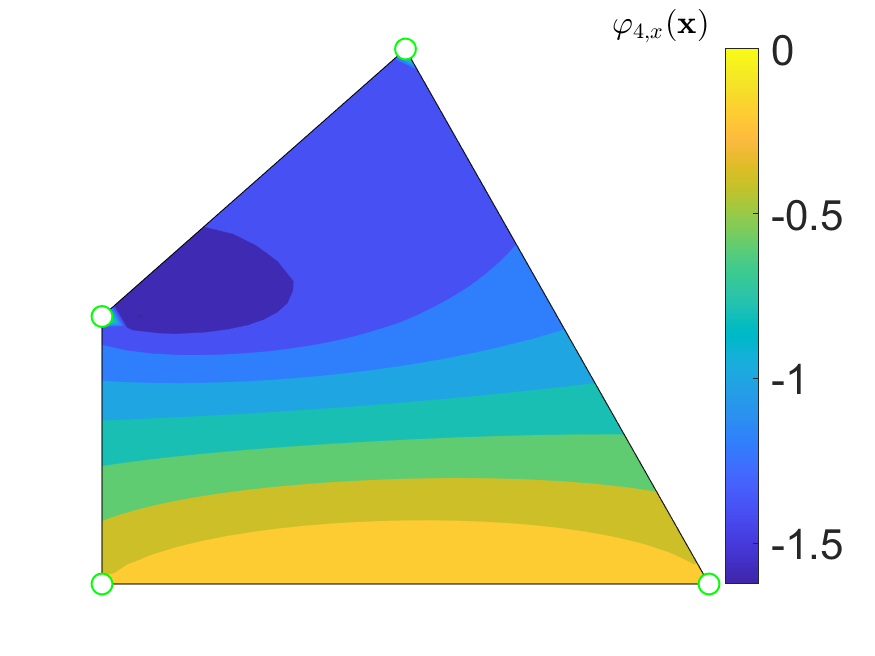}
    \end{subfigure}
    \caption{}
    \label{fig:exConv_mc_der_x}
\end{subfigure}
\begin{subfigure}[b]{\textwidth}
    \begin{subfigure}[b]{0.25\textwidth}
    \includegraphics[width=\textwidth]{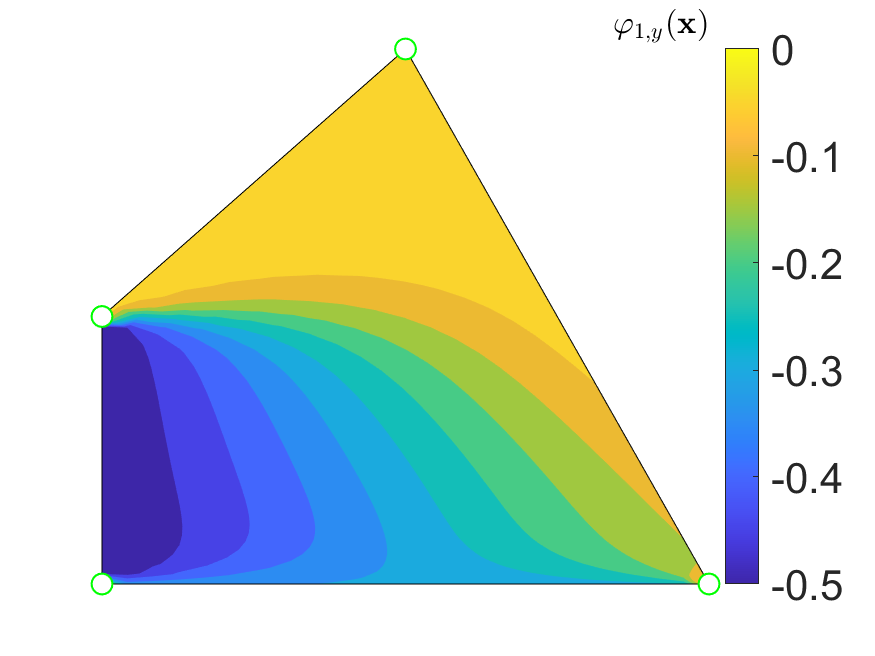}
    \end{subfigure}%
    \hfill
    \begin{subfigure}[b]{0.25\textwidth}
    \includegraphics[width=\textwidth]{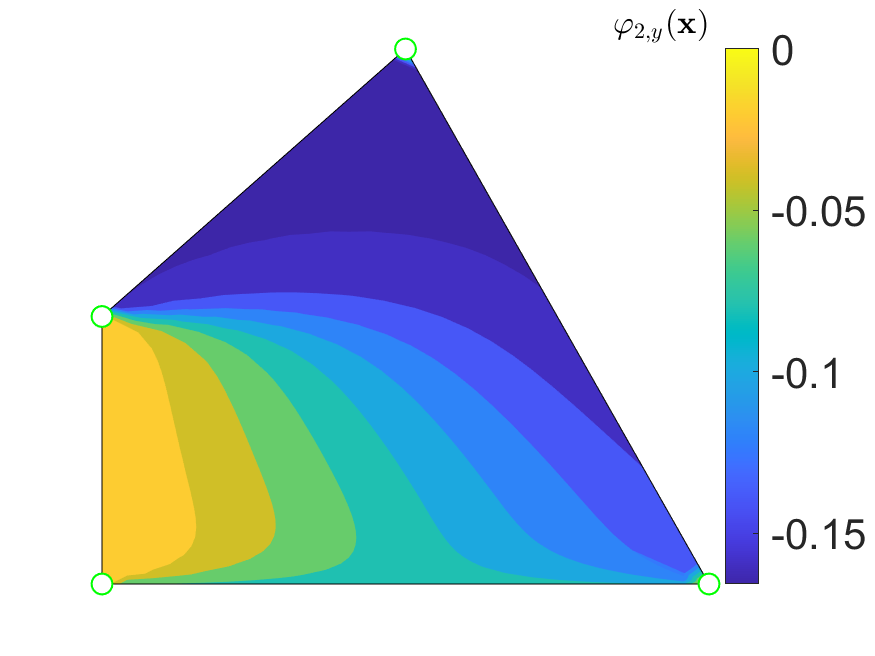}
    \end{subfigure}%
    \hfill
    \begin{subfigure}[b]{0.25\textwidth}
    \includegraphics[width=\textwidth]{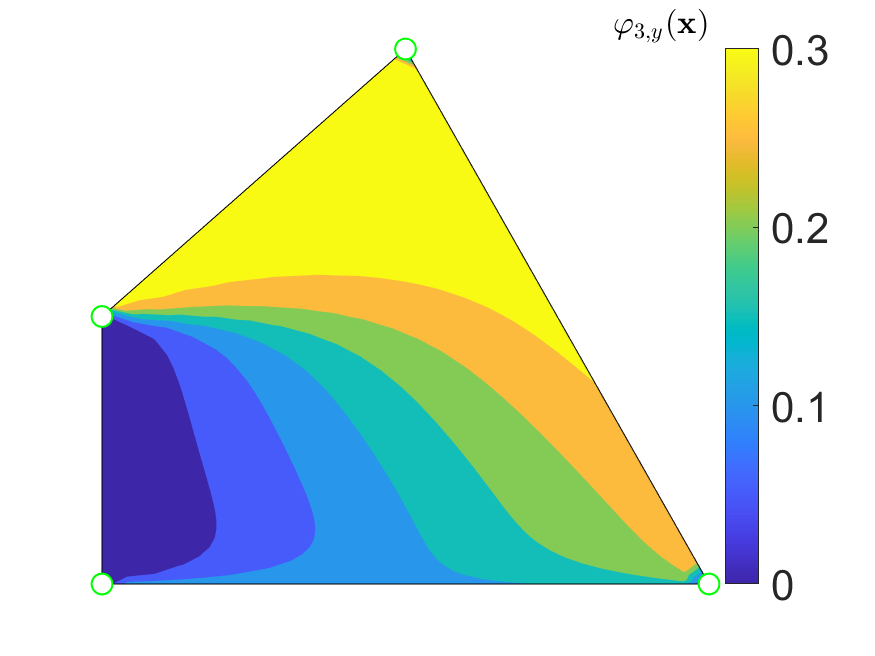}
    \end{subfigure}%
    \hfill
    \begin{subfigure}[b]{0.25\textwidth}
    \includegraphics[width=\textwidth]{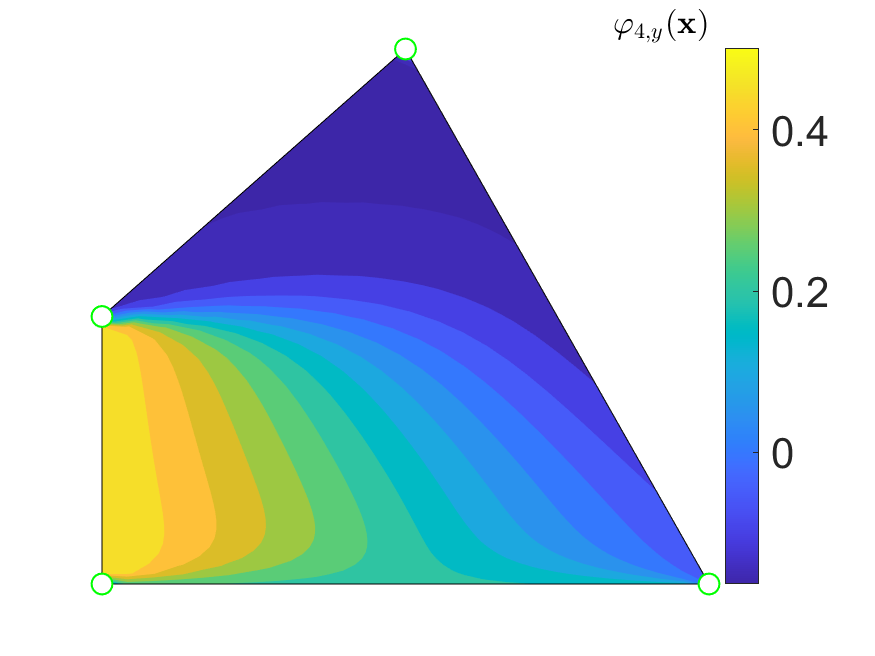}
    \end{subfigure}
    \caption{}
    \label{fig:exConv_mc_der_y}
\end{subfigure}
    \caption{Plots of (a) $x$-derivative and (b) $y$-derivative of moment coordinates $\varphi_i$ ($i=1,2,3,4$) on $Q$ for Example~\ref{ex:conv}.}
    \label{fig:exConv_mc_der}
\end{figure}
\begin{figure}
    \centering
\begin{subfigure}[b]{\textwidth}
    \begin{subfigure}[b]{0.25\textwidth}
    \includegraphics[width=\textwidth]{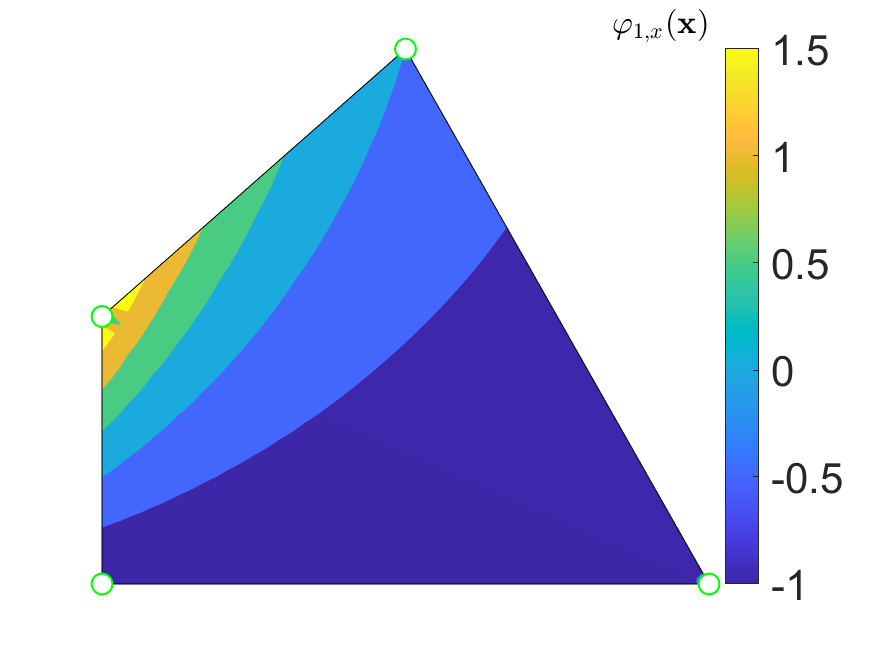}
    \end{subfigure}%
    \hfill
    \begin{subfigure}[b]{0.25\textwidth}
    \includegraphics[width=\textwidth]{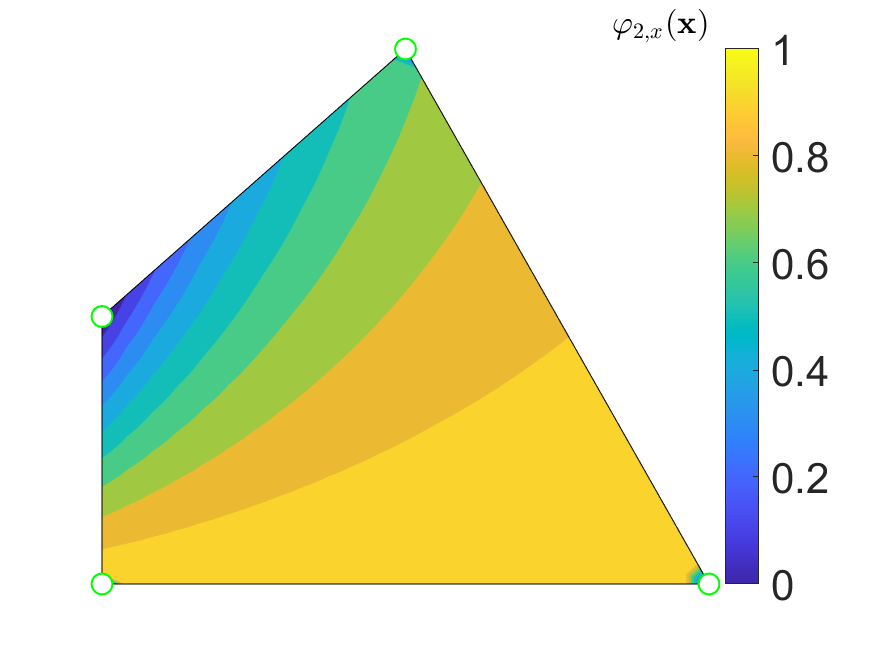}
    \end{subfigure}%
    \hfill
    \begin{subfigure}[b]{0.25\textwidth}
    \includegraphics[width=\textwidth]{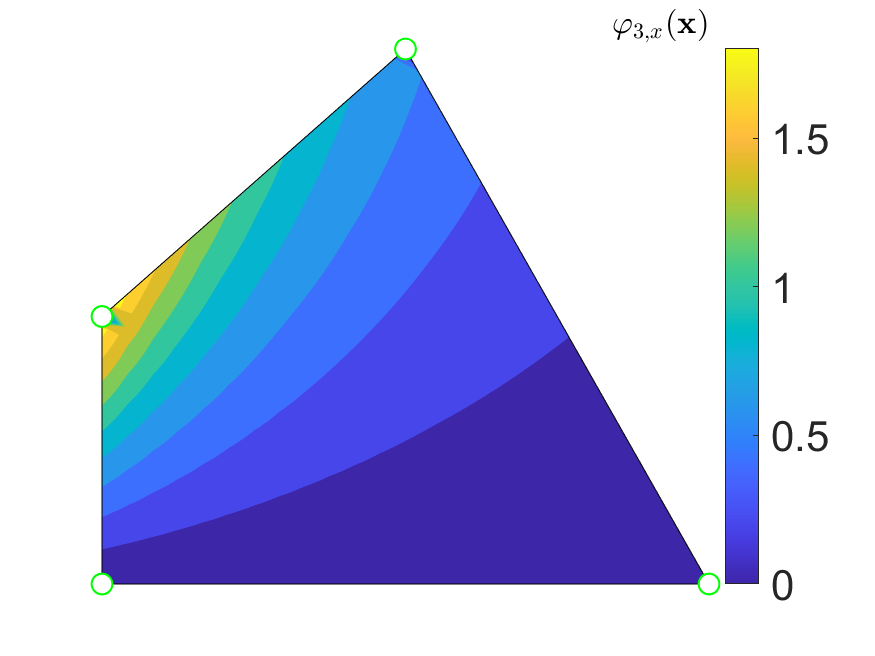}
    \end{subfigure}%
    \hfill
    \begin{subfigure}[b]{0.25\textwidth}
    \includegraphics[width=\textwidth]{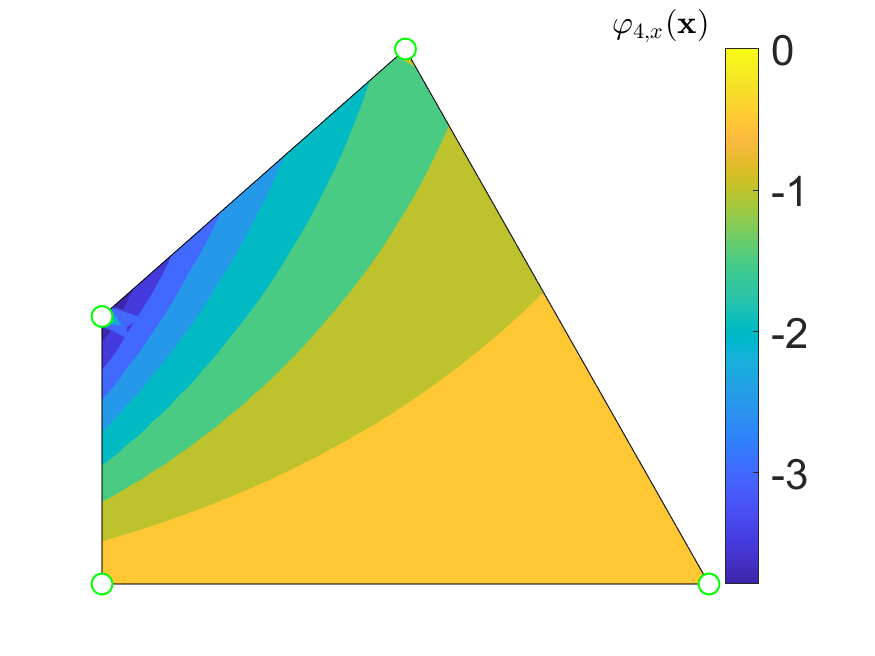}
    \end{subfigure}
    \caption{}
    \label{fig:exConv_wc_der_x}
\end{subfigure}
\begin{subfigure}[b]{\textwidth}
    \begin{subfigure}[b]{0.25\textwidth}
    \includegraphics[width=\textwidth]{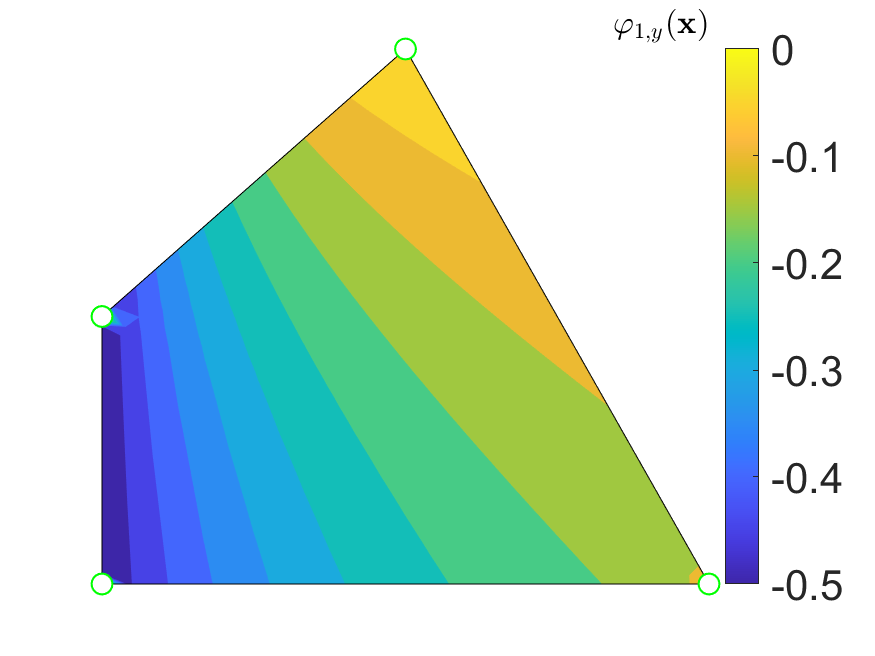}
    \end{subfigure}%
    \hfill
    \begin{subfigure}[b]{0.25\textwidth}
    \includegraphics[width=\textwidth]{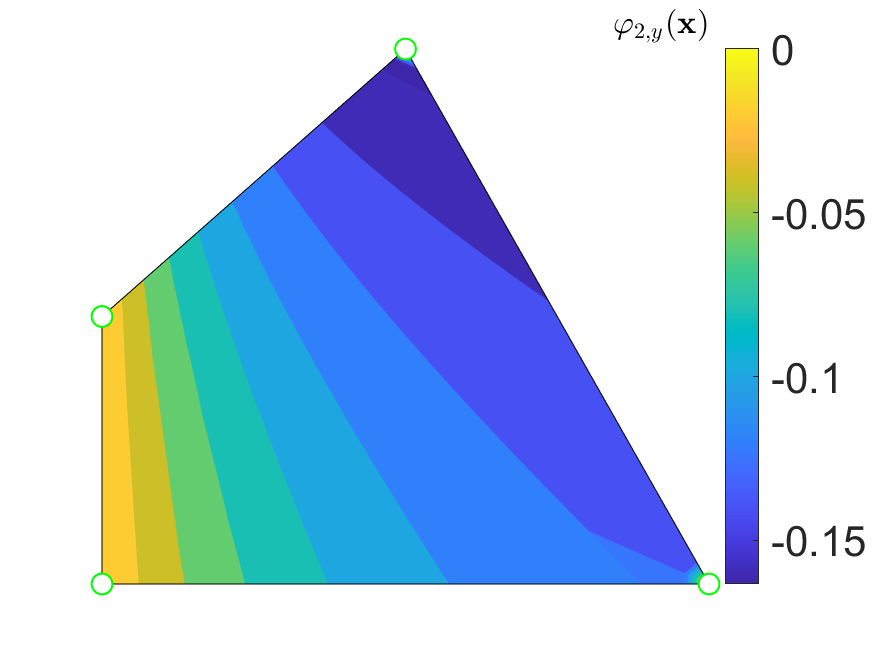}
    \end{subfigure}%
    \hfill
    \begin{subfigure}[b]{0.25\textwidth}
    \includegraphics[width=\textwidth]{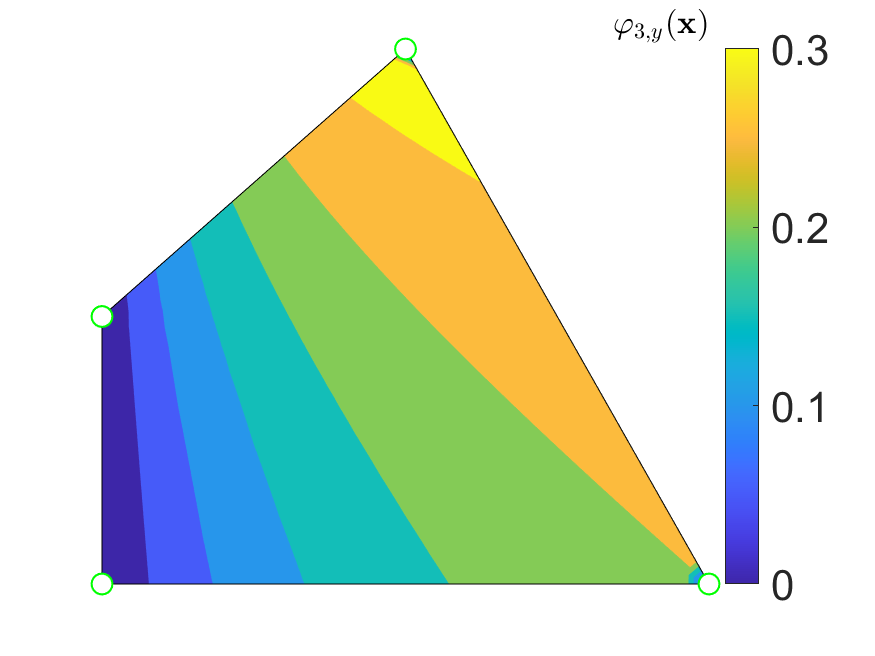}
    \end{subfigure}%
    \hfill
    \begin{subfigure}[b]{0.25\textwidth}
    \includegraphics[width=\textwidth]{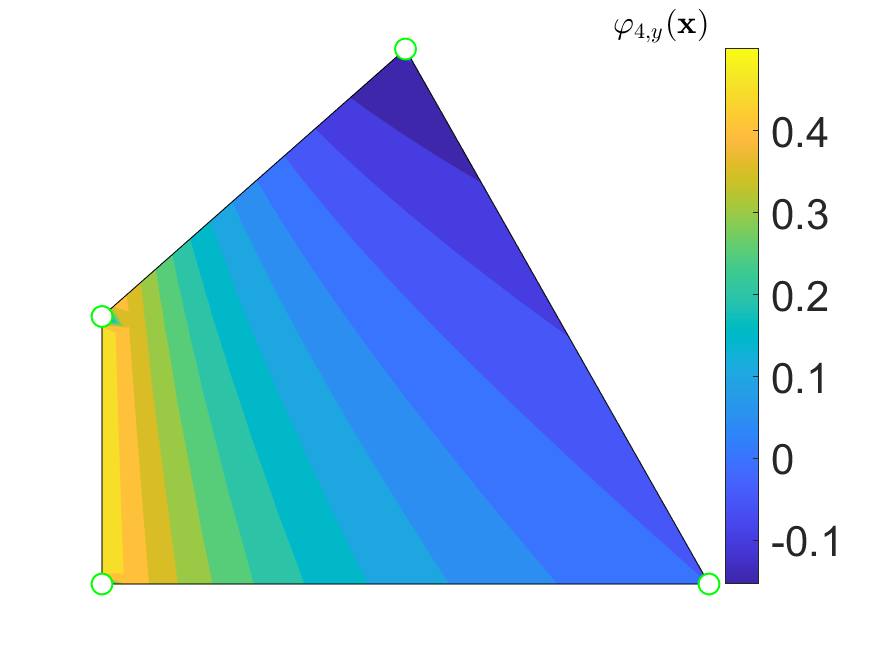}
    \end{subfigure}
    \caption{}
    \label{fig:exConv_wc_der_y}
\end{subfigure}
    \caption{Plots of (a) $x$-derivative and (b) $y$-derivative of Wachspress coordinates $\varphi_i$ ($i=1,2,3,4$) on $Q$ for Example~\ref{ex:conv}.}
    \label{fig:exConv_wc_der}
\end{figure}

\begin{exm}\label{ex:nonconv}
Let us consider the nonconvex quadrilateral $Q$ with vertices
\[
v_1=\bmat{0 \\ 0}, \ v_2=\bmat{2 \\ 0}, \ v_3=\bmat{1 \\ 4}, \ v_4=\bmat{1 \\ 2}.
\]
\end{exm}
On a nonconvex quadrilateral, Wachspress coordinates are not valid but moment (mean value) coordinates are permissible. In Figure~\ref{fig:exNonconv_mc} and Figure \ref{fig:exNonconv_mc_der}, moment coordinates and their partial derivatives, respectively, are presented for each vertex and once again we observe that these coordinates meet the desired properties of generalized barycentric coordinates given in~\eqref{eq:gbc}.
The analytical solutions for moment coordinates in this case are:

\[
\resizebox{\textwidth}{!}{$
\gbc=
\left[\begin{array}{c}
-\frac{2\,\sqrt{x^2-4\,x+y^2+4}-8\,\sqrt{x^2-2\,x+y^2-4\,y+5}-4\,\sqrt{x^2-2\,x+y^2-8\,y+17}-2\,x\,\sqrt{x^2-4\,x+y^2+4}+4\,x\,\sqrt{x^2-2\,x+y^2-4\,y+5}+2\,x\,\sqrt{x^2-2\,x+y^2-8\,y+17}+y\,\sqrt{x^2-2\,x+y^2-4\,y+5}+y\,\sqrt{x^2-2\,x+y^2-8\,y+17}}{2\,\left(4\,\sqrt{x^2-2\,x+y^2-4\,y+5}+2\,\sqrt{x^2-2\,x+y^2-8\,y+17}-\sqrt{x^2-4\,x+y^2+4}+\sqrt{x^2+y^2}\right)}\\
\frac{2\,x\,\sqrt{x^2+y^2}-2\,\sqrt{x^2+y^2}+4\,x\,\sqrt{x^2-2\,x+y^2-4\,y+5}+2\,x\,\sqrt{x^2-2\,x+y^2-8\,y+17}-y\,\sqrt{x^2-2\,x+y^2-4\,y+5}-y\,\sqrt{x^2-2\,x+y^2-8\,y+17}}{2\,\left(4\,\sqrt{x^2-2\,x+y^2-4\,y+5}+2\,\sqrt{x^2-2\,x+y^2-8\,y+17}-\sqrt{x^2-4\,x+y^2+4}+\sqrt{x^2+y^2}\right)}\\
\frac{2\,x\,\sqrt{x^2-4\,x+y^2+4}-y\,\sqrt{x^2-4\,x+y^2+4}+2\,x\,\sqrt{x^2+y^2}+y\,\sqrt{x^2+y^2}-4\,\sqrt{x^2+y^2}+2\,y\,\sqrt{x^2-2\,x+y^2-4\,y+5}}{2\,\left(4\,\sqrt{x^2-2\,x+y^2-4\,y+5}+2\,\sqrt{x^2-2\,x+y^2-8\,y+17}-\sqrt{x^2-4\,x+y^2+4}+\sqrt{x^2+y^2}\right)}\\
-\frac{4\,x\,\sqrt{x^2-4\,x+y^2+4}-y\,\sqrt{x^2-4\,x+y^2+4}+4\,x\,\sqrt{x^2+y^2}+y\,\sqrt{x^2+y^2}-8\,\sqrt{x^2+y^2}-2\,y\,\sqrt{x^2-2\,x+y^2-8\,y+17}}{2\,\left(4\,\sqrt{x^2-2\,x+y^2-4\,y+5}+2\,\sqrt{x^2-2\,x+y^2-8\,y+17}-\sqrt{x^2-4\,x+y^2+4}+\sqrt{x^2+y^2}\right)}
\end{array}\right].
$
}
\]
\begin{figure}
    \centering
    \begin{subfigure}[b]{0.25\textwidth}
    \includegraphics[width=\textwidth]{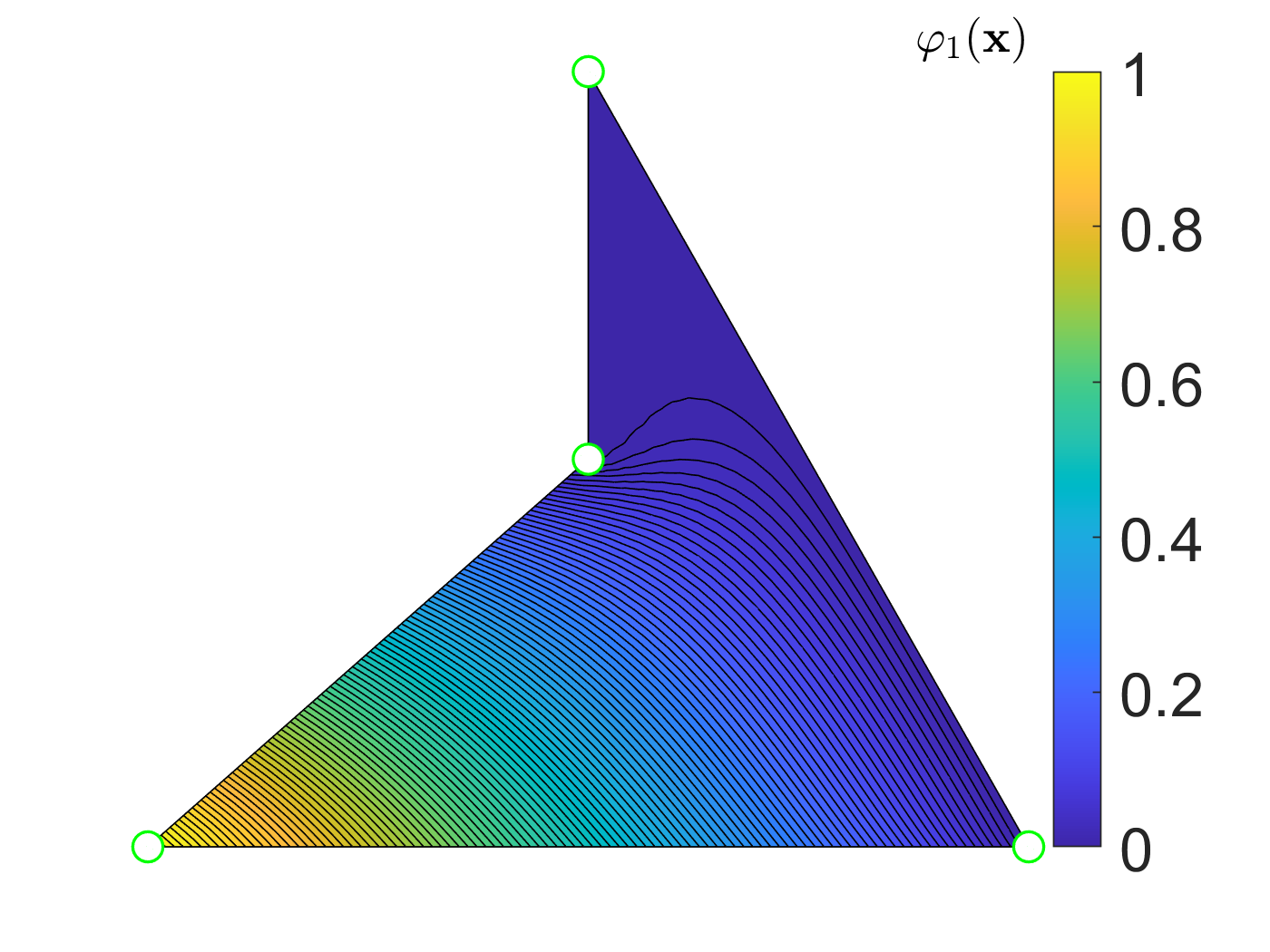}
    \end{subfigure}%
    \hfill
    \begin{subfigure}[b]{0.25\textwidth}
    \includegraphics[width=\textwidth]{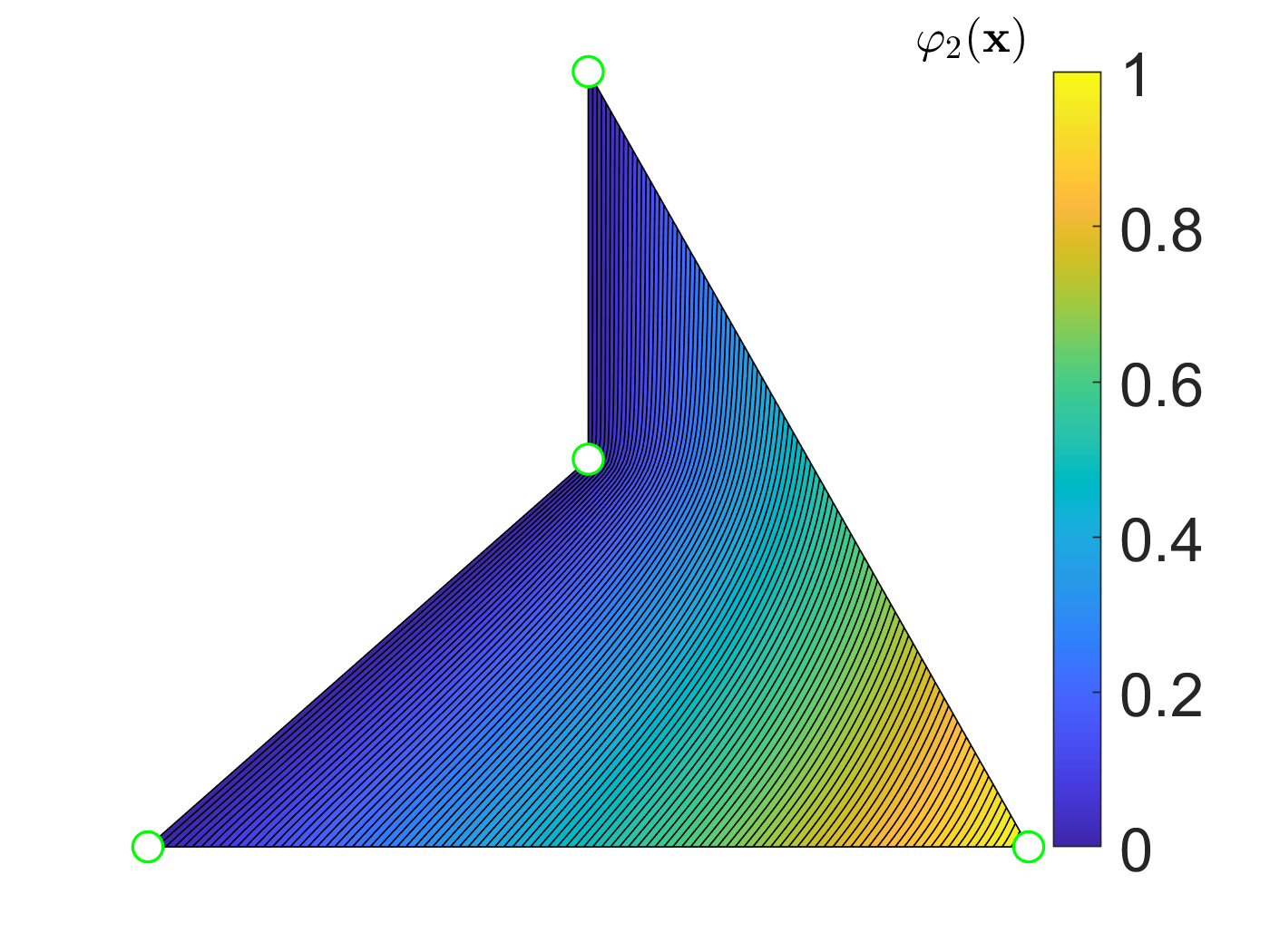}
    \end{subfigure}%
    \hfill
    \begin{subfigure}[b]{0.25\textwidth}
    \includegraphics[width=\textwidth]{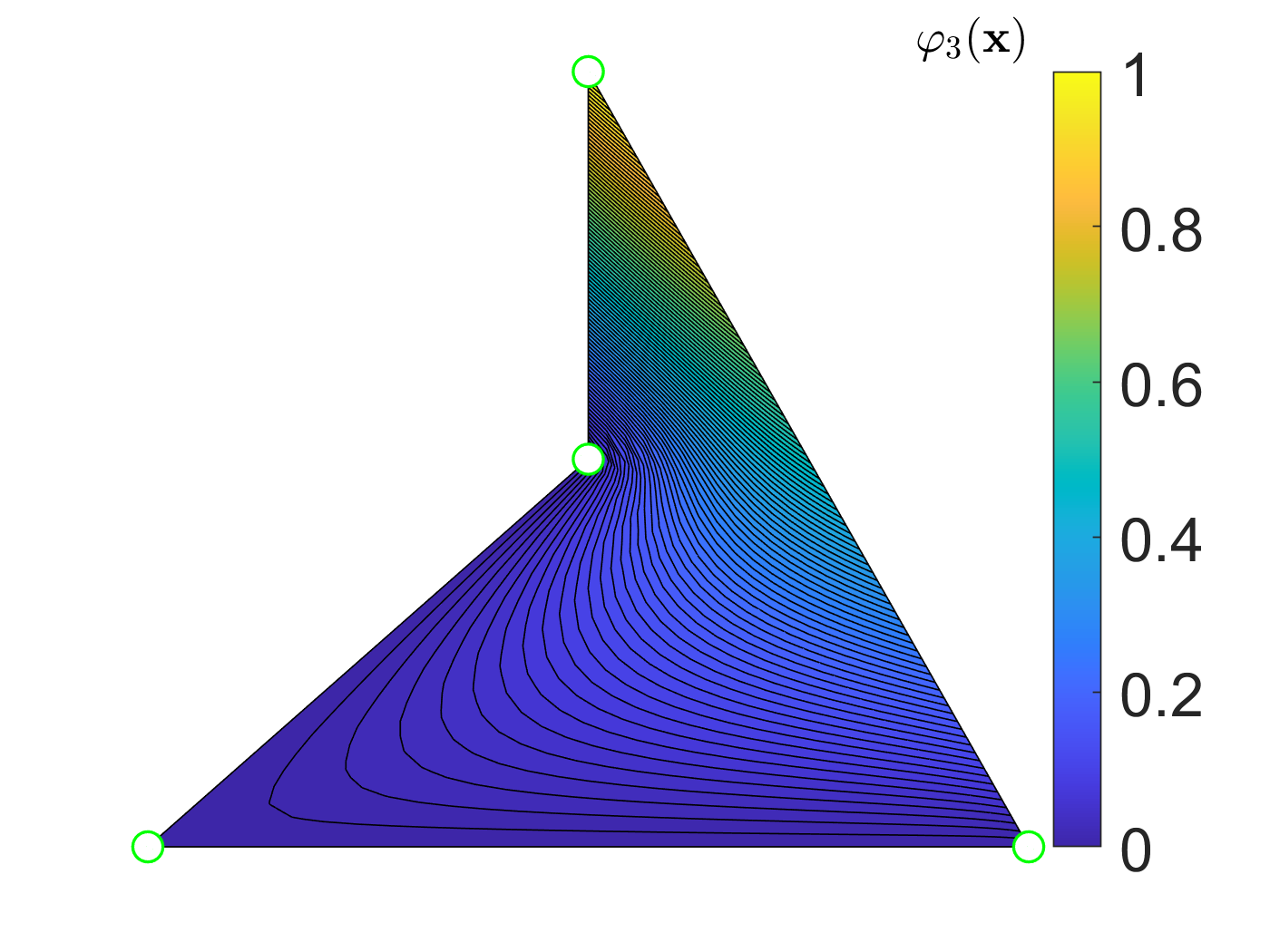}
    \end{subfigure}%
    \hfill
    \begin{subfigure}[b]{0.25\textwidth}
    \includegraphics[width=\textwidth]{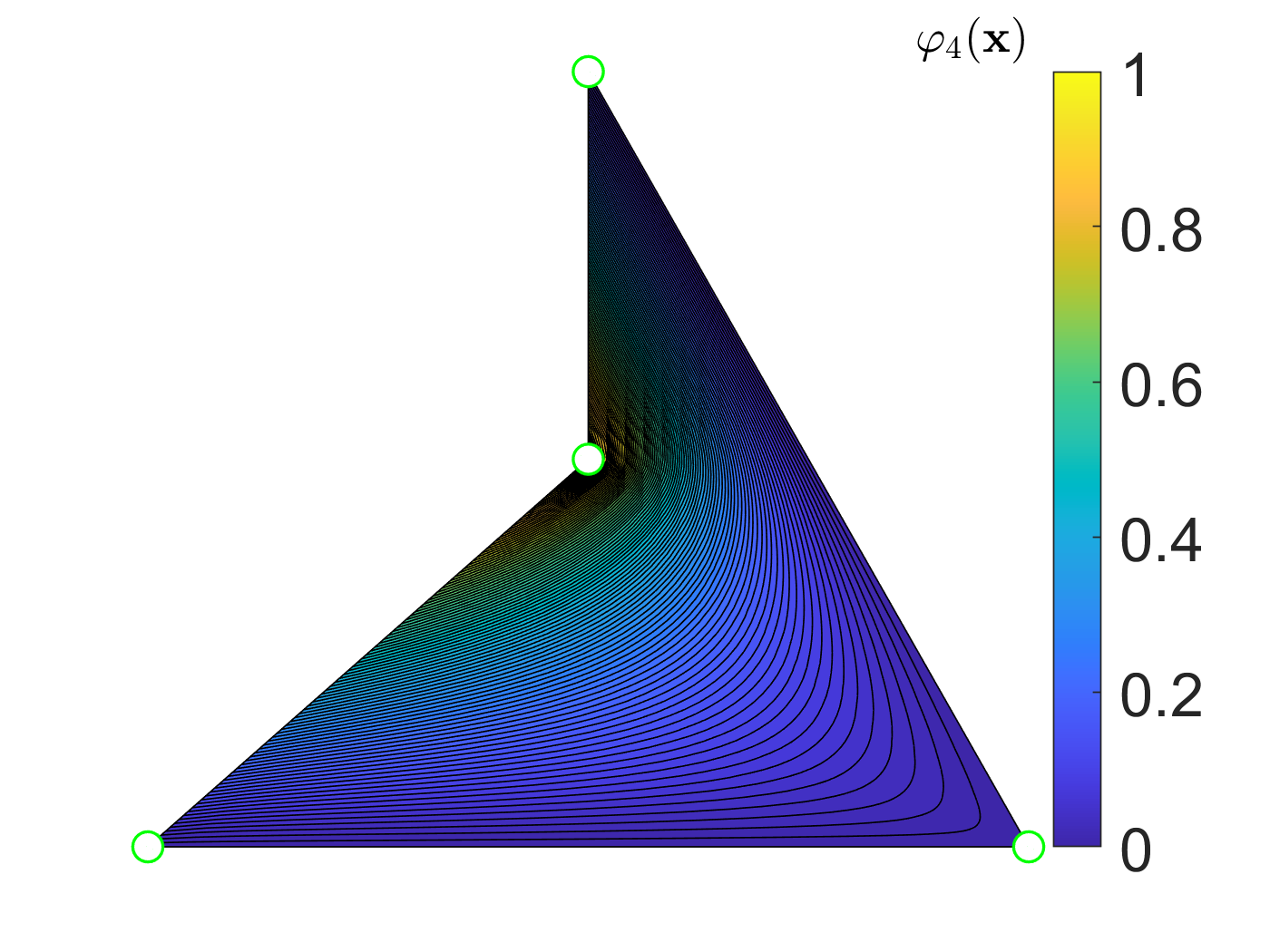}
    \end{subfigure}
\caption{Moment coordinates on $Q$ for Example~\ref{ex:nonconv}.}
\label{fig:exNonconv_mc}
\end{figure}
%nonconv mc der
\begin{figure}
    \centering
\begin{subfigure}[b]{\textwidth}
    \begin{subfigure}[b]{0.25\textwidth}
    \includegraphics[width=\textwidth]{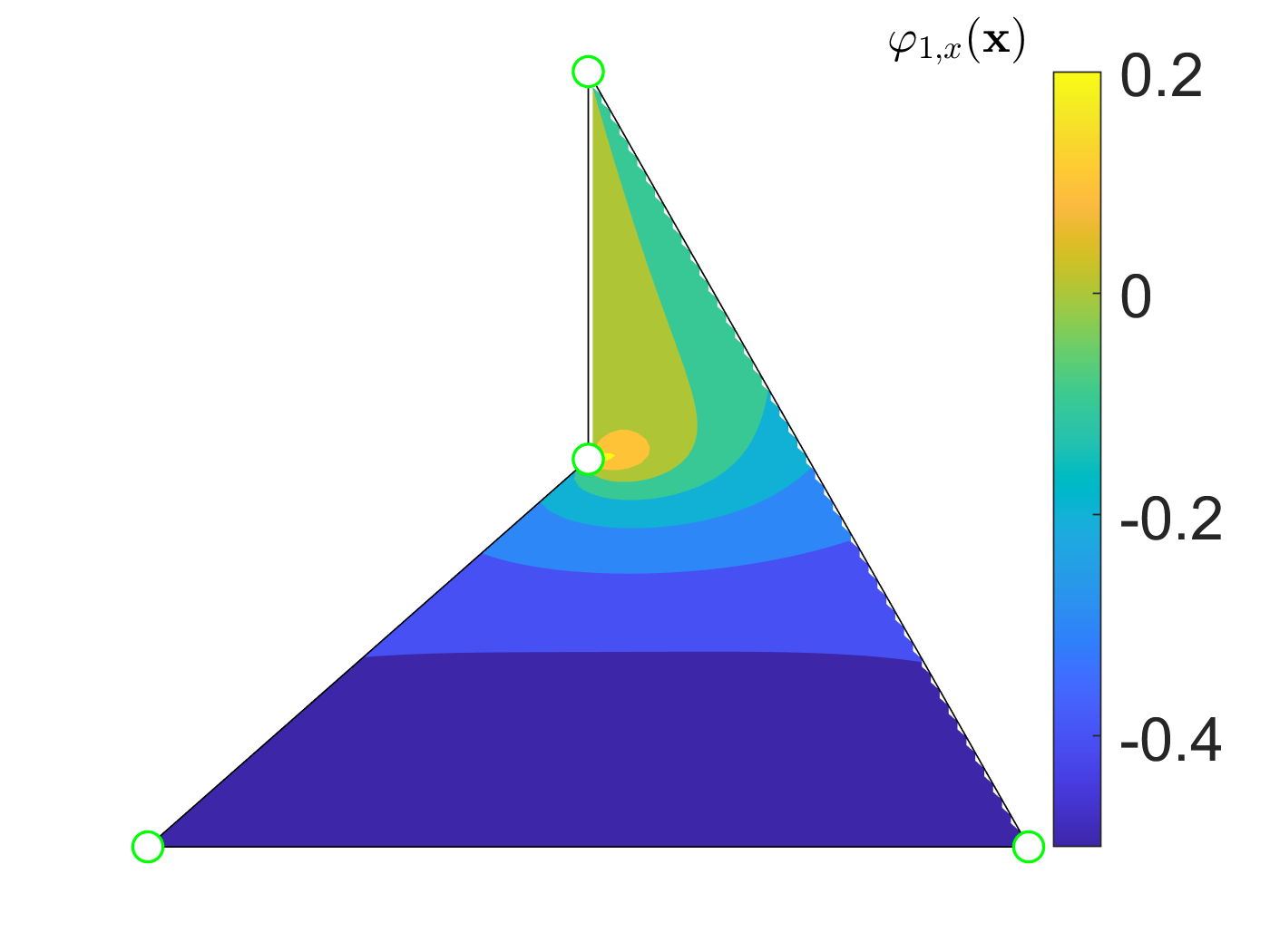}
    \end{subfigure}%
    \hfill
    \begin{subfigure}[b]{0.25\textwidth}
    \includegraphics[width=\textwidth]{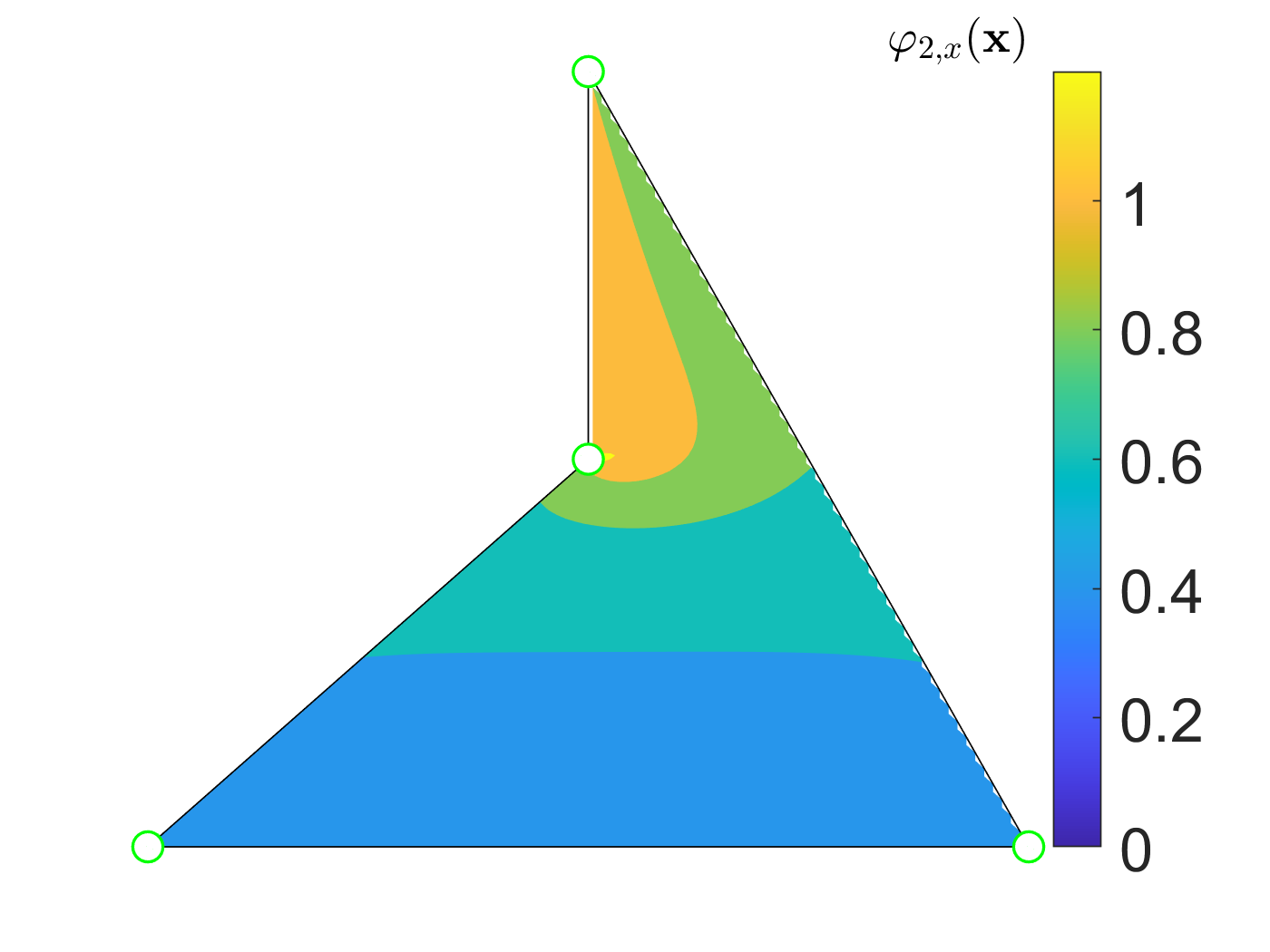}
    \end{subfigure}%
    \hfill
    \begin{subfigure}[b]{0.25\textwidth}
    \includegraphics[width=\textwidth]{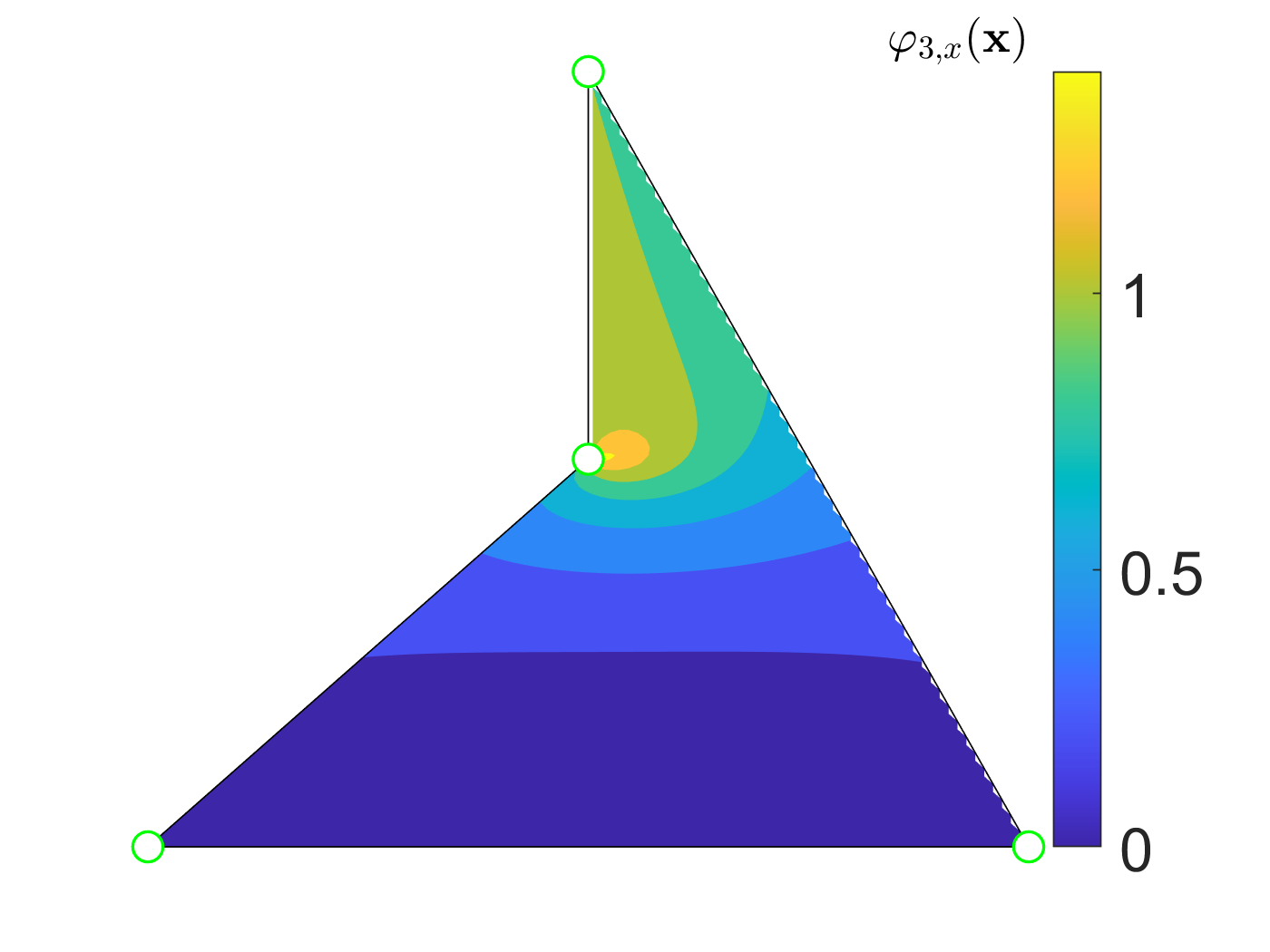}
    \end{subfigure}%
    \hfill
    \begin{subfigure}[b]{0.25\textwidth}
    \includegraphics[width=\textwidth]{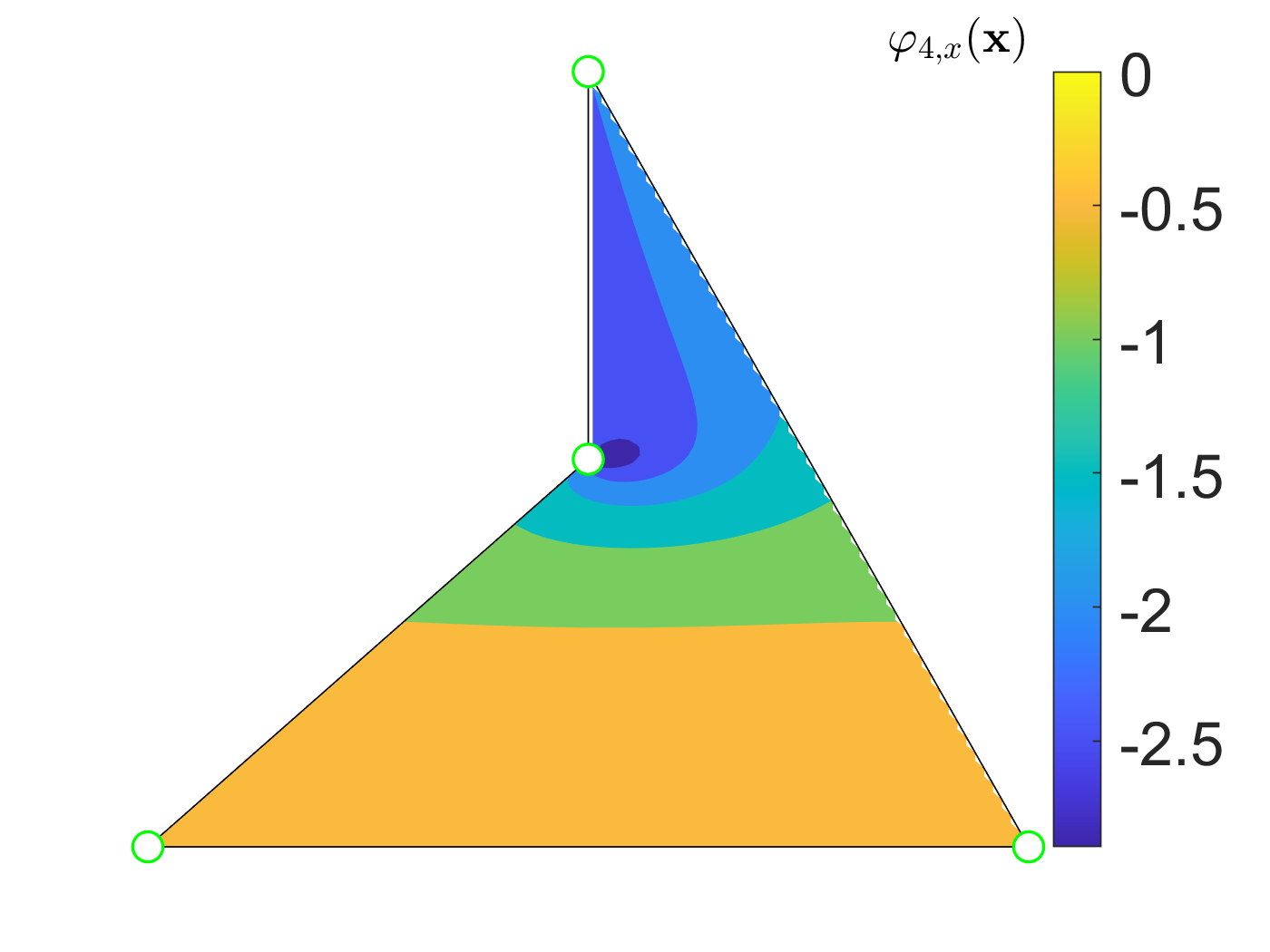}
    \end{subfigure}
    % \caption{Moment Coordinates relative to $Q$.}
    \caption{}
    \label{fig:exNonconv_mc_der_x}
\end{subfigure}
\begin{subfigure}[b]{\textwidth}
    \begin{subfigure}[b]{0.25\textwidth}
    \includegraphics[width=\textwidth]{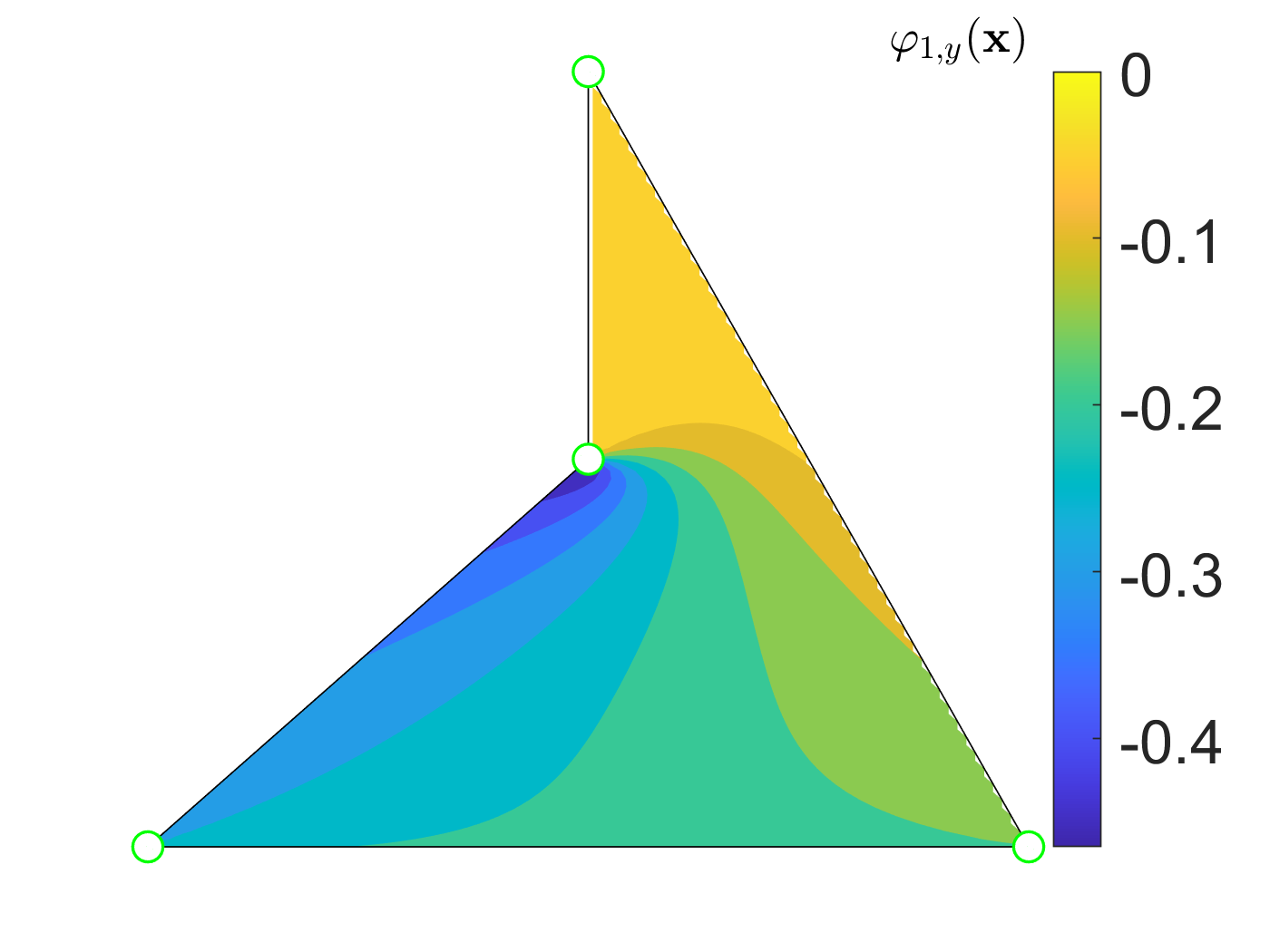}
    \end{subfigure}%
    \hfill
    \begin{subfigure}[b]{0.25\textwidth}
    \includegraphics[width=\textwidth]{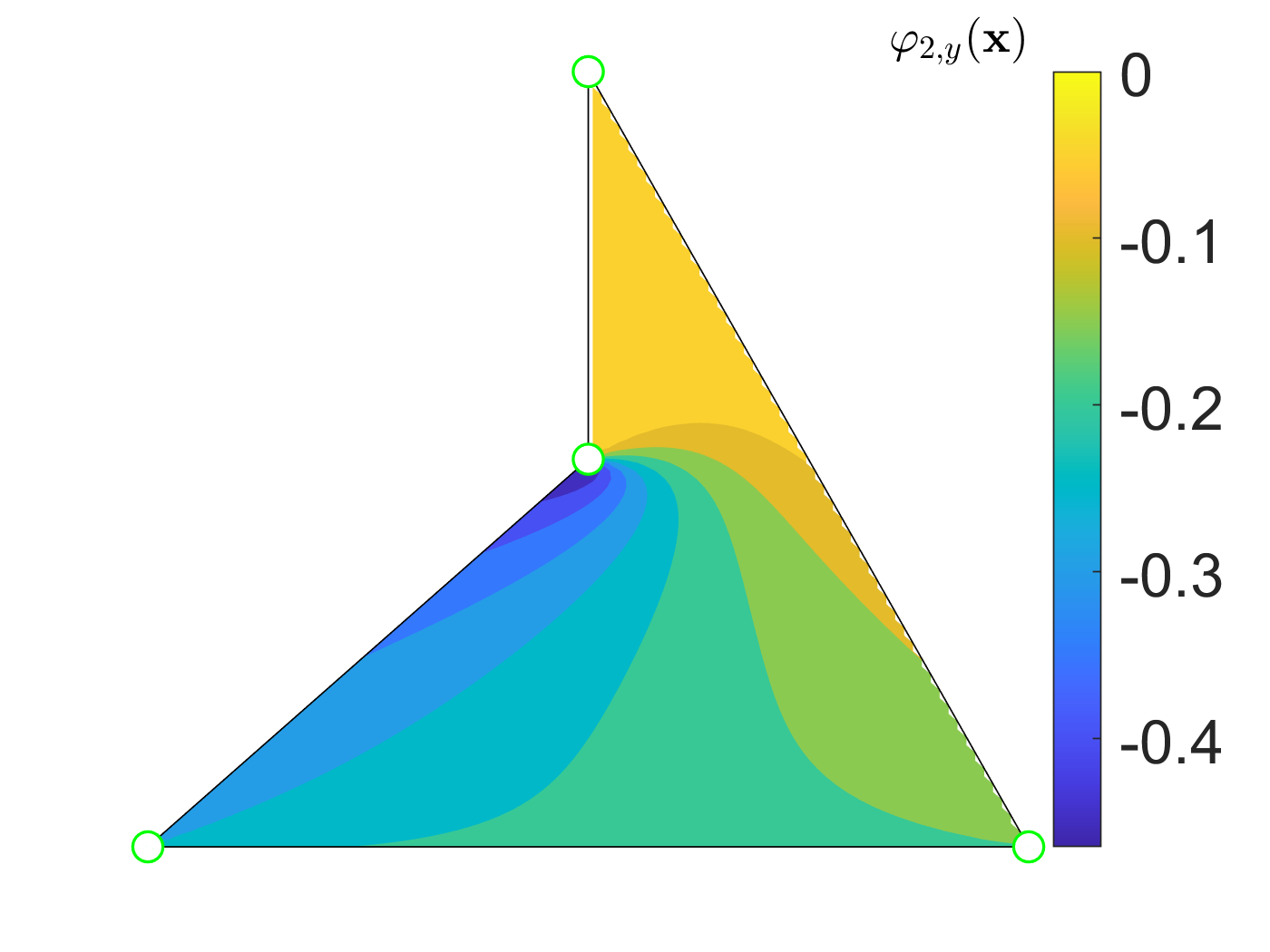}
    \end{subfigure}%
    \hfill
    \begin{subfigure}[b]{0.25\textwidth}
    \includegraphics[width=\textwidth]{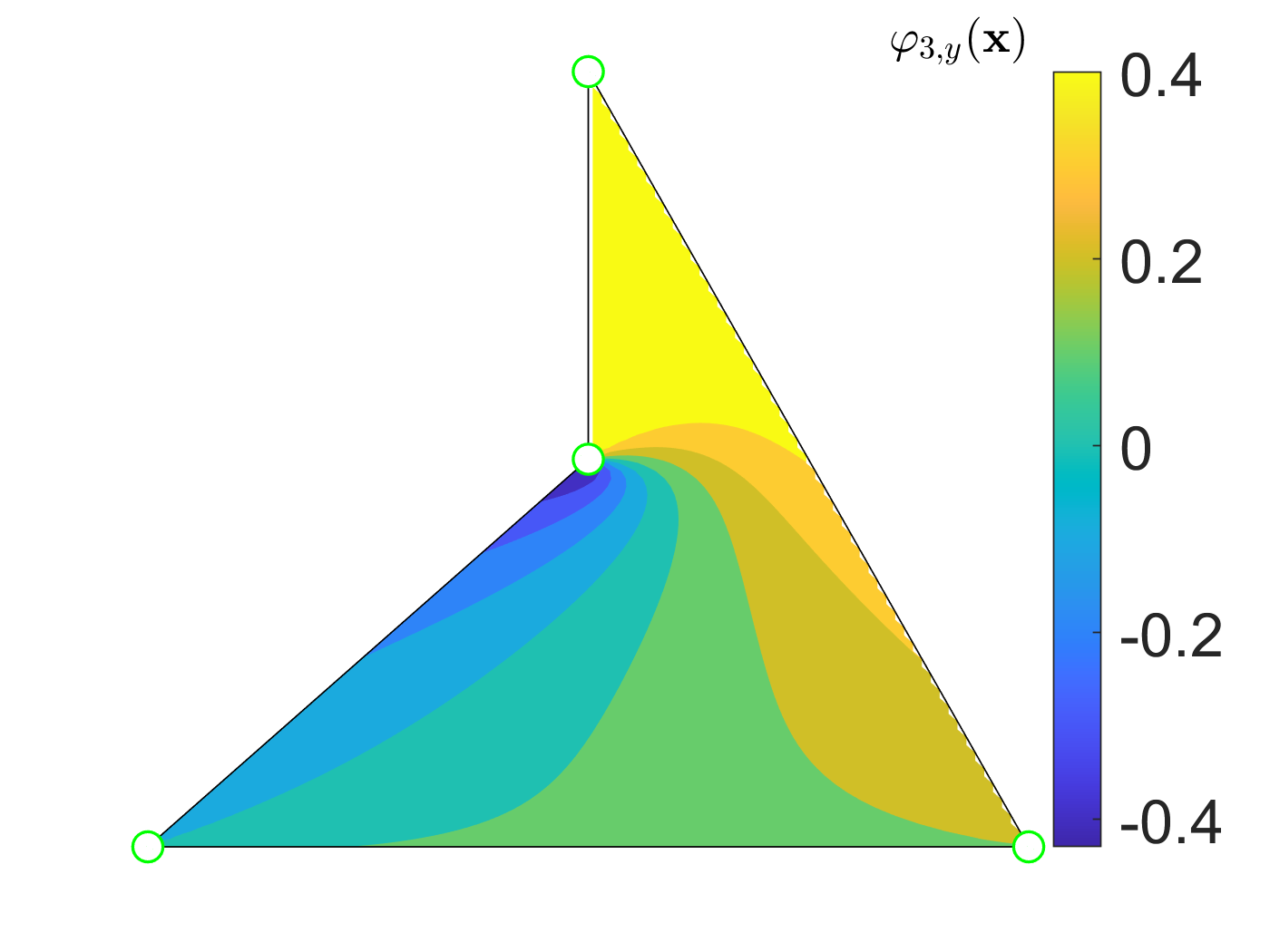}
    \end{subfigure}%
    \hfill
    \begin{subfigure}[b]{0.25\textwidth}
    \includegraphics[width=\textwidth]{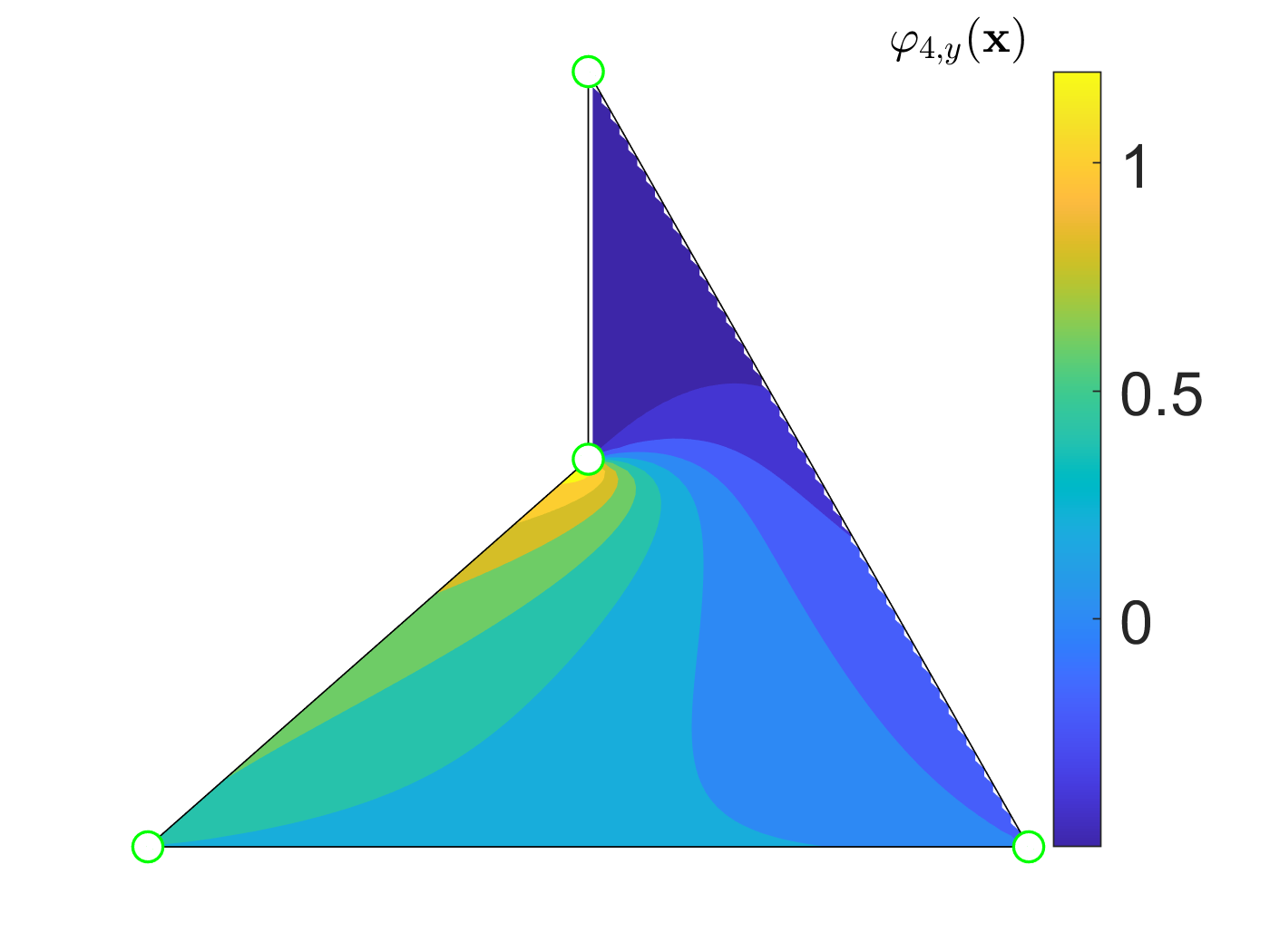}
    \end{subfigure}
    % \caption{Wachspress Coordinates relative to $Q$.}
    \caption{}
    \label{fig:exNonconv_mc_der_y}
\end{subfigure}
    \caption{Plots of (a) $x$-derivative and (b) $y$-derivative of moment coordinate $\varphi_i$ ($i=1,2,3,4$) on $Q$ for Example~\ref{ex:nonconv}.}
    \label{fig:exNonconv_mc_der}
\end{figure}

Moment coordinates are also well defined and provide the expected solutions in the case of a degenerate quadrilateral that has two consecutive collinear sides, as we show in the next example.
\begin{exm}\label{ex:conv_degenerate}
Let us consider the convex quadrilateral $Q$ with vertices
\[
v_1=\bmat{0 \\ 0}, \ v_2=\bmat{1 \\ 0}, \ v_3=\bmat{2 \\ 0}, \ v_4=\bmat{0 \\ 1},
\]
so that $Q$ has the shape of a triangle.
\end{exm}
In Figure \ref{fig:exConv_degenerate}, moment coordinates are presented for each vertex, whose analytical expressions are given by
\[
\resizebox{\textwidth}{!}{$
\varphi=
\left[\begin{array}{c} -\frac{x\,\sqrt{x^2-2\,x+y^2+1}+x\,\sqrt{x^2-4\,x+y^2+4}+2\,y\,\sqrt{x^2-2\,x+y^2+1}+y\,\sqrt{x^2-4\,x+y^2+4}-y\,\sqrt{x^2+y^2-2\,y+1}-2\,\sqrt{x^2-2\,x+y^2+1}-\sqrt{x^2-4\,x+y^2+4}}{2\,\sqrt{x^2-2\,x+y^2+1}+\sqrt{x^2-4\,x+y^2+4}+\sqrt{x^2+y^2}}\\ -\frac{2\,y\,\sqrt{x^2+y^2-2\,y+1}-x\,\sqrt{x^2-4\,x+y^2+4}+x\,\sqrt{x^2+y^2}+2\,y\,\sqrt{x^2+y^2}-2\,\sqrt{x^2+y^2}}{2\,\sqrt{x^2-2\,x+y^2+1}+\sqrt{x^2-4\,x+y^2+4}+\sqrt{x^2+y^2}}\\ \frac{x\,\sqrt{x^2-2\,x+y^2+1}+y\,\sqrt{x^2+y^2-2\,y+1}+x\,\sqrt{x^2+y^2}+y\,\sqrt{x^2+y^2}-\sqrt{x^2+y^2}}{2\,\sqrt{x^2-2\,x+y^2+1}+\sqrt{x^2-4\,x+y^2+4}+\sqrt{x^2+y^2}}\\ y \end{array}\right].
$
}
\]
As expected from the geometry of the example, the last component $\varphi_4$ is a linear function.
\begin{figure}
    \centering
    \begin{subfigure}[b]{0.25\textwidth}
    \includegraphics[width=\textwidth]{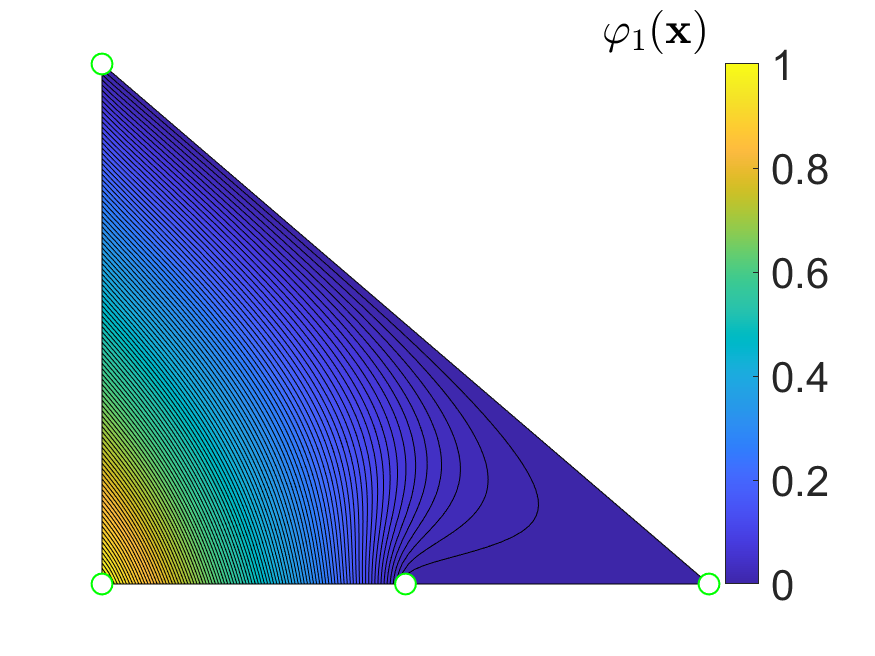}
    \end{subfigure}%
    \hfill
    \begin{subfigure}[b]{0.25\textwidth}
    \includegraphics[width=\textwidth]{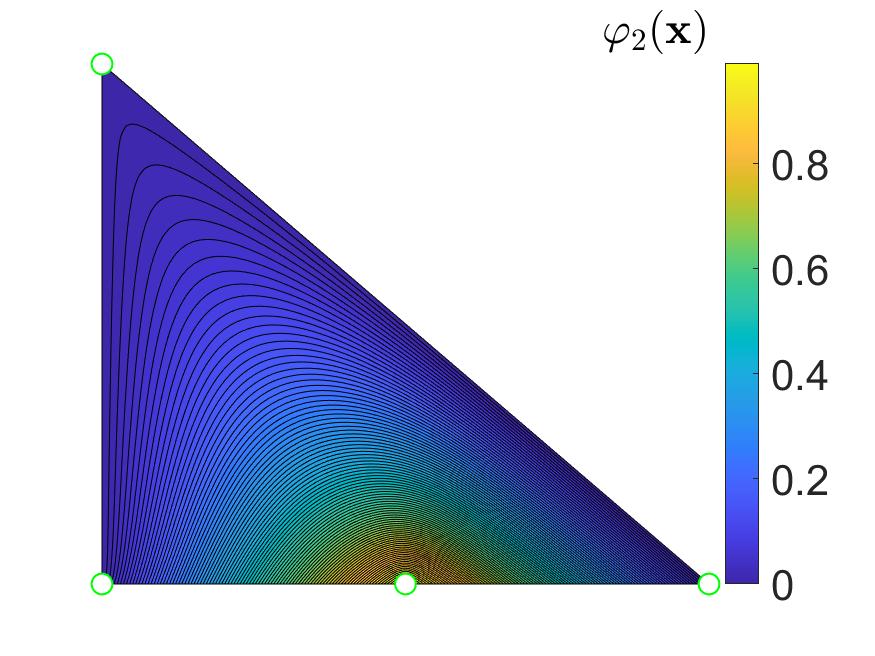}
    \end{subfigure}%
    \hfill
    \begin{subfigure}[b]{0.25\textwidth}
    \includegraphics[width=\textwidth]{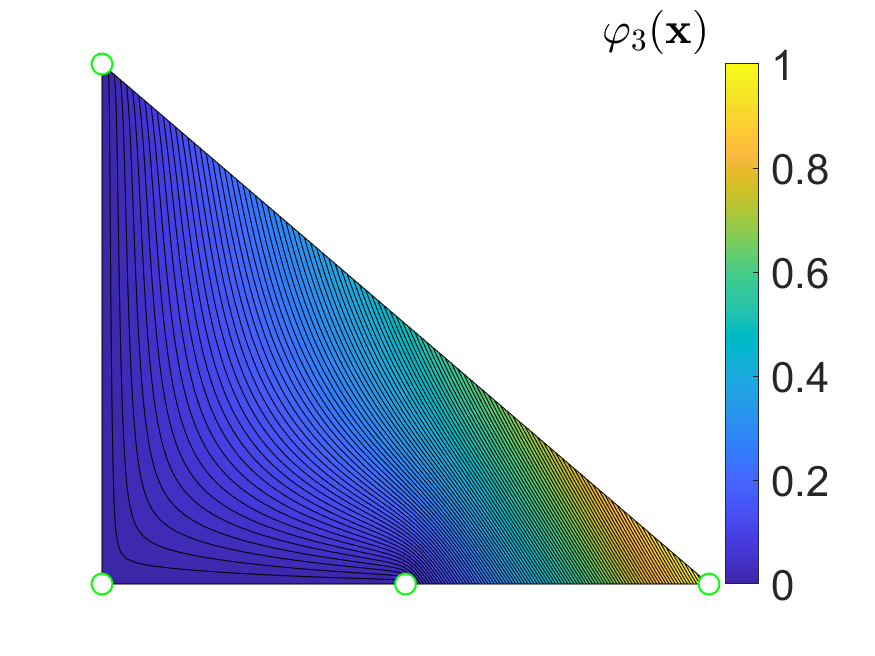}
    \end{subfigure}%
    \hfill
    \begin{subfigure}[b]{0.25\textwidth}
    \includegraphics[width=\textwidth]{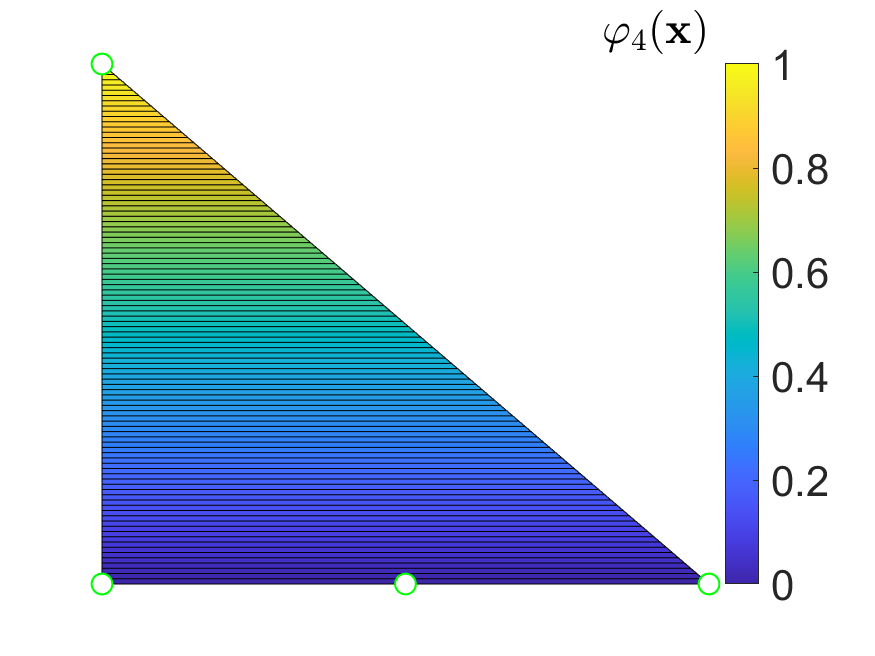}
    \end{subfigure}
\caption{Moment coordinates on $Q$ for Example~\ref{ex:conv_degenerate}.}
\label{fig:exConv_degenerate}
\end{figure}

\begin{rem}\label{rem:cond}
It is worth noticing that there is a potential ill-conditioning of the linear system \eqref{eq:momSystem}, but this is a mere reflection of the fact that the barycentric coordinates are not necessarily stable under perturbation. Indeed, this is not a peculiarity of the system formulation. \\
For example, let us consider the nonconvex quadrilateral $Q_1$ with vertices
\[
v_1 =\bmat{0 \\ 0},\; v_2 = \bmat{1 \\ 1},\; v_3 = \bmat{0 \\ 10^{-8}},\; v_4 =\bmat{-1 \\ 1}.
\]
Perturbing randomly $Q_1$ to order $10^{-9}$ produces a variation of order $10^{-2}$ in the barycentric coordinates of the point $P=\begin{bmatrix} 0 \\ \frac12 10^{-8} \end{bmatrix}$. \\
Further, let us consider the rectangle $Q_2$ with vertices
\[
v_1 =\bmat{0 \\ 0},\; v_2 = \bmat{1 \\ 0},\; v_3 = \bmat{1 \\ 10^{-8}},\; v_4 =\bmat{0 \\ 10^{-8}},
\]
where $\frac{\overline{v_1v_4}}{\overline{v_1v_2}}=10^{-8}$, so that $\overline{v_1v_4}\ll \overline{v_1v_2}$.
With $P=\begin{bmatrix} \frac12 \\ \frac12 10^{-8} \end{bmatrix}$, one obtains $\varphi=\frac14\1$; however, a random perturbation of the vertices of order $10^{-8}$ creates a variation in the barycentric coordinates of order $1$, regardless of the inexact arithmetic.
\end{rem}

\subsection{Hexahedron}
\begin{exm}\label{ex:convHex}
Let us consider the convex hexahedron $H_c$ with vertices
\[
\resizebox{\textwidth}{!}{$
v_1=\bmat{1 \\ 2 \\ 1}, \ v_2=\bmat{1 \\ 2 \\ -1}, \ v_3=\bmat{1 \\ 0 \\ -1}, \ v_4=\bmat{1 \\ 0 \\ 1},
v_5=\bmat{-1 \\ 1 \\ 1}, \ v_6=\bmat{-1 \\ 1 \\ -1}, \ v_7=\bmat{-1 \\ -1 \\ -1}, \ v_8=\bmat{-1 \\ -1 \\ 1}
$
}
\]
\end{exm}
Moment coordinates are nonnegative on the whole hexahedron, piecewise affine on the boundary edges and satisfy the Kronecker-delta property at the vertices; in addition, they satisfy the facet-reduction property on each face, as shown in Figure~\ref{fig:convHex}. Furthermore, they provide results that are distinct from 3D mean value coordinates.
\begin{figure}
    \centering
    \begin{subfigure}[b]{0.25\textwidth}
    \includegraphics[width=\textwidth]{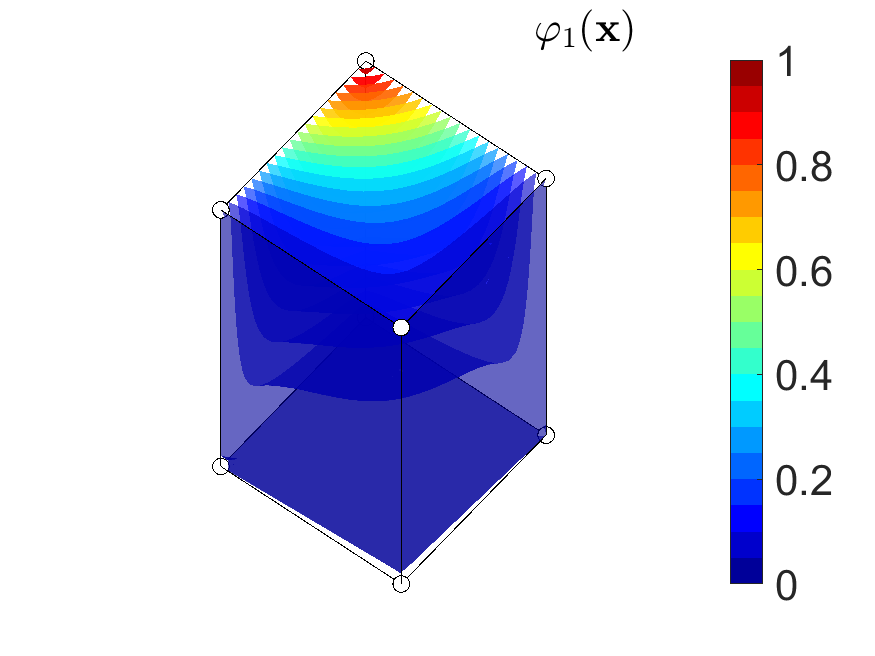}
    \end{subfigure}%
    \hfill
    \begin{subfigure}[b]{0.25\textwidth}
    \includegraphics[width=\textwidth]{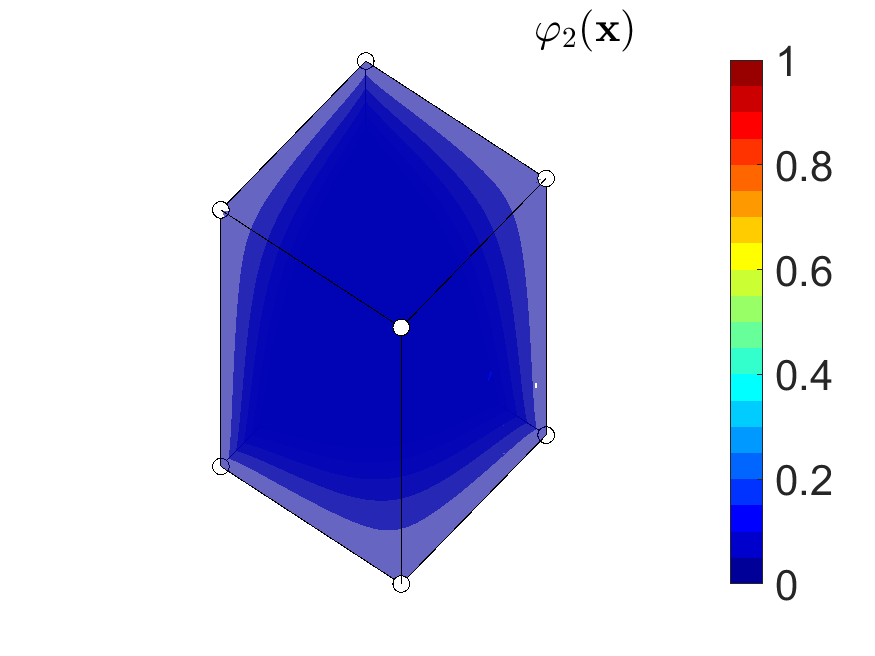}
    \end{subfigure}%
    \hfill
    \begin{subfigure}[b]{0.25\textwidth}
    \includegraphics[width=\textwidth]{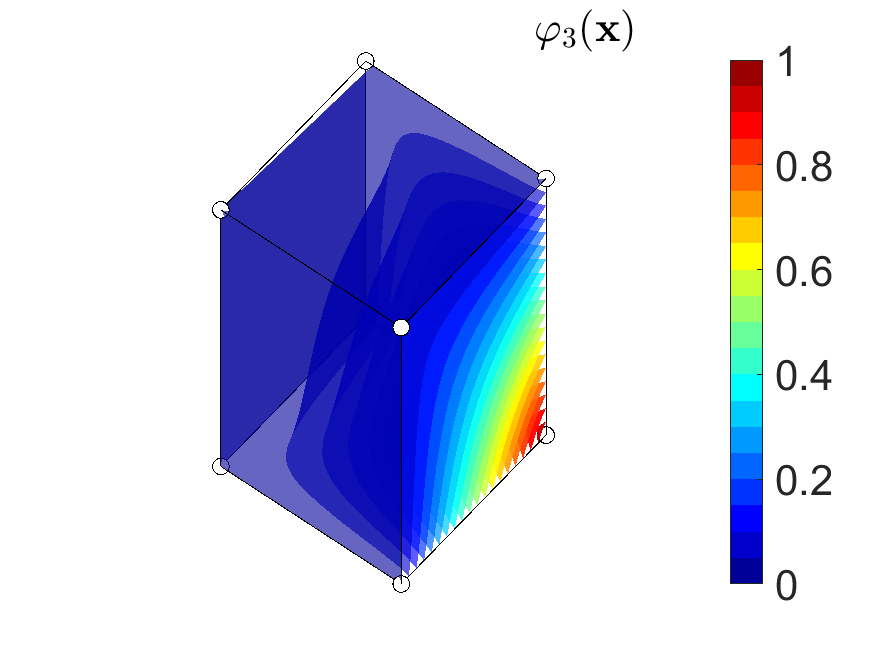}
    \end{subfigure}%
    \hfill
    \begin{subfigure}[b]{0.25\textwidth}
    \includegraphics[width=\textwidth]{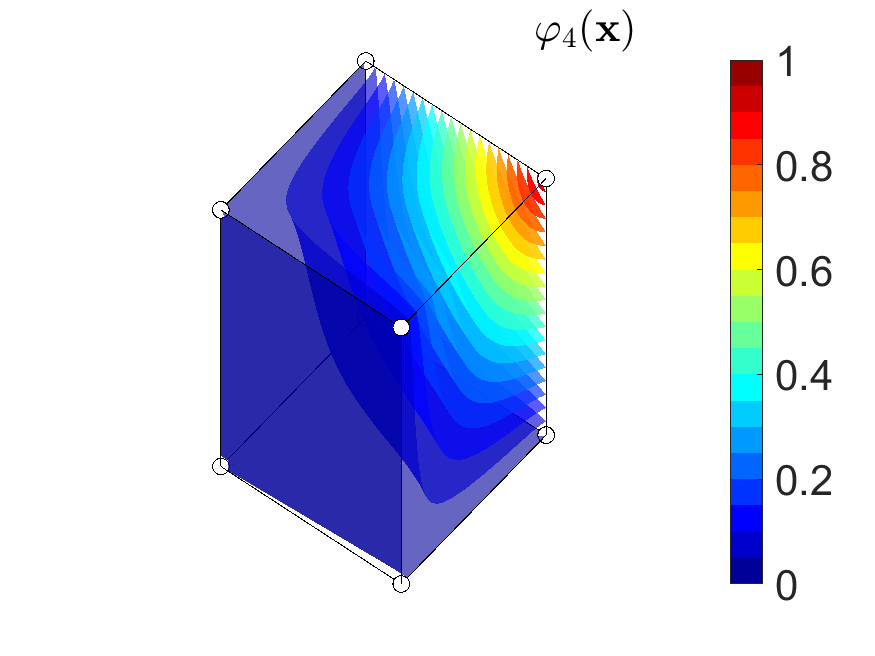}
    \end{subfigure}
    \begin{subfigure}[b]{0.25\textwidth}
    \includegraphics[width=\textwidth]{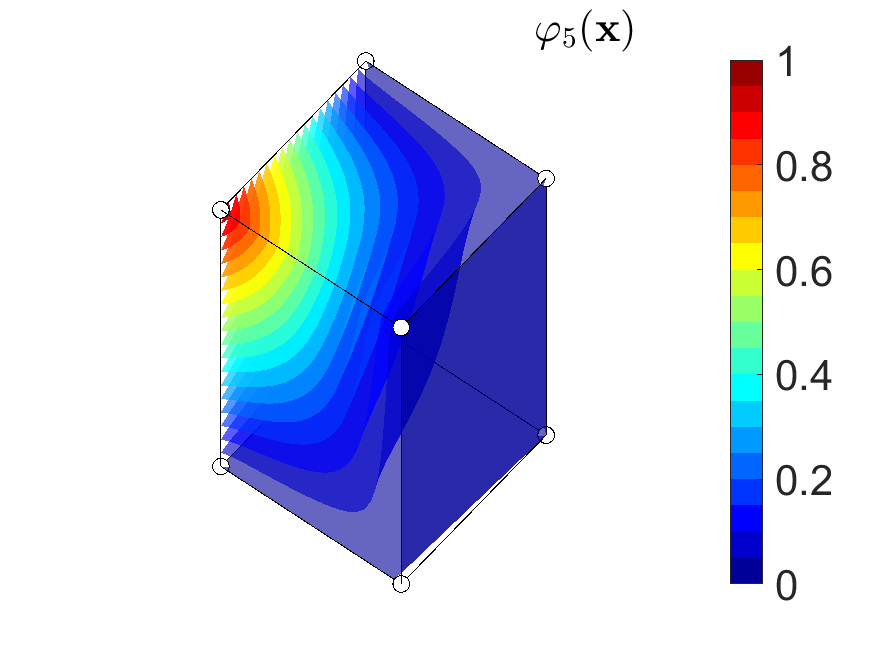}
    \end{subfigure}%
    \hfill
    \begin{subfigure}[b]{0.25\textwidth}
    \includegraphics[width=\textwidth]{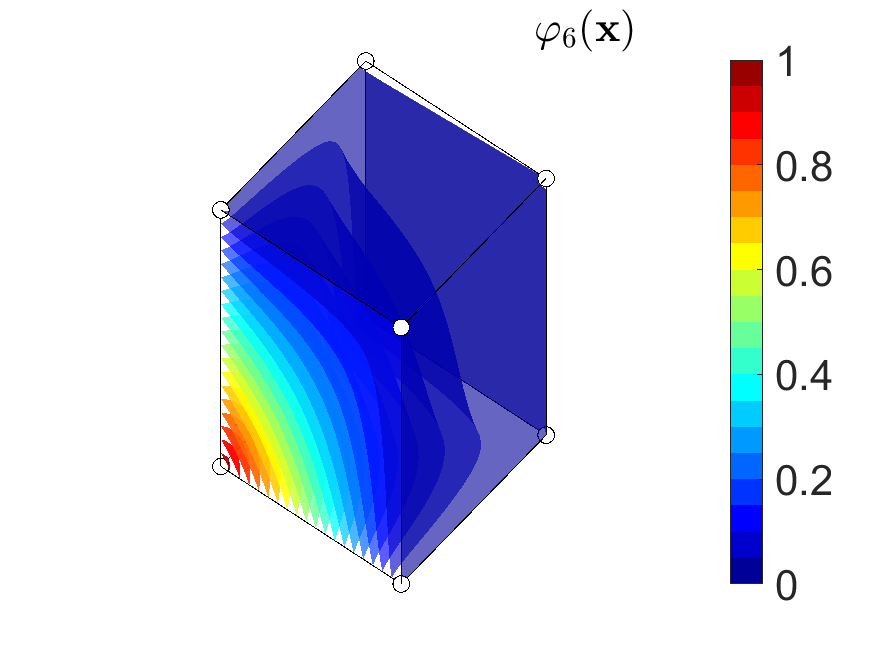}
    \end{subfigure}%
    \hfill
    \begin{subfigure}[b]{0.25\textwidth}
    \includegraphics[width=\textwidth]{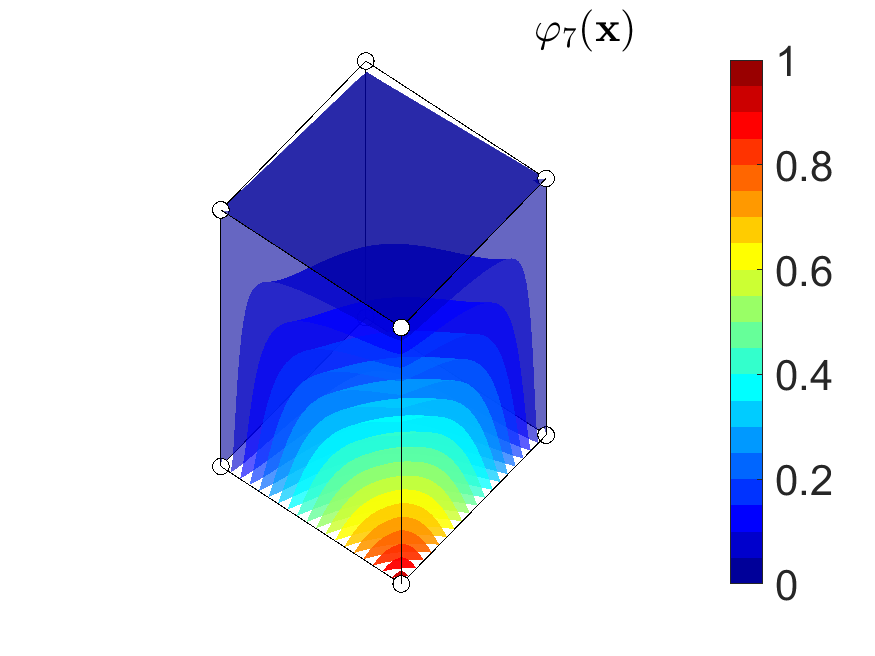}
    \end{subfigure}%
    \hfill
    \begin{subfigure}[b]{0.25\textwidth}
    \includegraphics[width=\textwidth]{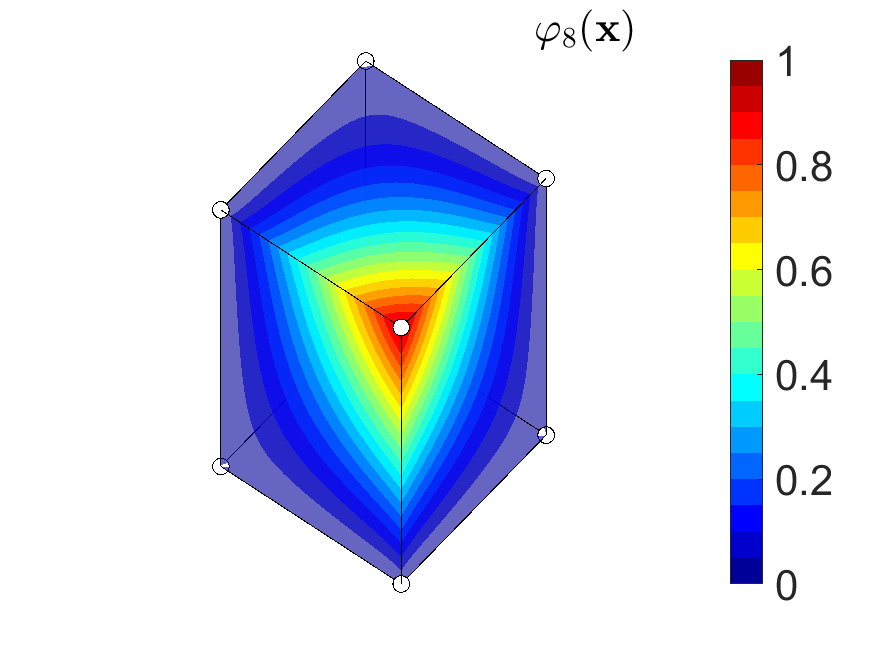}
\end{subfigure}
\caption{Moment coordinates on $H_c$ for Example~\ref{ex:convHex}.}
\label{fig:convHex}
\end{figure}

\section{Conclusions}\label{sec:conclusions}
In this paper, we introduced moment coordinates, $\gbc(p)$, a new generalized barycentric coordinates over common finite element geometries. Generalized barycentric coordinates must satisfy the linear reproducing conditions and it is desirable that they be nonnegative. We supplemented the reproducing conditions by
moment regularizing equality constraints so that the linear system of equations yielded $\gbc(p)$ that were unique and componentwise nonnegative. This idea of adding moment equality conditions goes back to work on sliding mode control in piecewise smooth dynamical systems~\cite{dd2:2017,dd4:2017}, in which one
searches for Filippov solutions in presence of dynamics along discontinuity manifolds of codimension 2 or higher. We showed that by judicious choice of the regularizing constraints, we could recover mean value coordinates~\cite{Floater:2003:MVC} on nonconvex quadrilaterals as well as Wachspress coordinates~\cite{Wachspress:2016:RBG} on convex quadrilaterals. Moment coordinates on general convex hexahedron were constructed by using signed partial distances~\cite{dd4:2017} in tandem with a local reference system, which ensured that
$V - p$ ($V$ is the matrix of vertex coordinates) had the
same sign pattern for all points in the hexahedron. Theoretical proofs were presented which were supported by analytical expressions for moment coordinates and their plots over quadrilaterals and hexahedra.

Extension of the present work to polygonal finite elements in two dimensions and to other finite element geometries (for example, prisms and pyramids) in three dimensions are of interest.
For a convex polytope with $n$ vertices, maximum-entropy
regularization~\cite{Sukumar:2004:COP} is a suitable approach to
obtain $\gbc(p)$, but it is not aligned with our objective of
obtaining a closed-form expression for $\gbc(p)$.
In keeping with the
approach presented in this paper, it
is desirable to devise $n$ linear constraints that yield
a feasible nonnegative solution for $\gbc(p)$. To this end, as a first step in future work we plan to study the extension of moment coordinates to simple polygons.

\section*{Acknowledgments}
FVD has been supported by \textit{REFIN} Project, grant number 812E4967, and by INdAM-GNCS 2023 Project, grant number CUP$\_$E53C22001930001.

\section*{Conflict of interest}
The authors declare no conflict of interest.

\end{document}